\documentclass[notitlepage]{amsart}

\usepackage{amsfonts}
\usepackage[latin1]{inputenc}
\usepackage[all]{xy} 
\usepackage{latexsym,amssymb,amscd}
\usepackage{amsmath}
\usepackage{mathrsfs}
\usepackage{multicol}

\usepackage{pb-diagram, pb-xy}

\usepackage{graphicx}
\usepackage{fancyhdr}
\usepackage{color}
\usepackage[T1]{fontenc}

\usepackage{tikz}
\usepackage{tikz-cd}
\usetikzlibrary{arrows}
 \usetikzlibrary{snakes}
 \usetikzlibrary{shapes}

\usetikzlibrary{patterns,backgrounds}

\newtheorem{pro}{Proposition}[section]
\newtheorem{teo}[pro]{Theorem}
\newtheorem{lem}[pro]{Lemma}
\newtheorem{defi}[pro]{Definition}
\newtheorem{cor}[pro]{Corollary}
\newtheorem{rk}[pro]{Remark}
\newtheorem{ex}[pro]{Example}

\def\U{\mathcal{U}}
\newcommand\restr[2]{{
		\left.\kern-\nulldelimiterspace 
		#1 
		\vphantom{\big|} 
		\right|_{#2} 
}}

\newcommand{\Ind}{\mathrm{Ind}}

\newcommand{\T}{\mathcal{T}}

\newcommand{\A}{\mathcal{A}}
\newcommand{\B}{\mathcal{B}}
\newcommand{\C}{\mathcal{C}}

\newcommand{\ltt}{\mathcal}

\def\temp{&} \catcode`&=\active \let&=\temp

\usepackage{latexsym,amssymb,amscd}
\usepackage{amsmath}

\begin{document}
\title[Triangular Matrix Categories]{Triangular Matrix Categories Over Path Categories and Quasi-hereditary Categories, as well as One-Point Extensions by projectives}
\author{R. Ochoa, M. Ort\'iz}
\thanks{2000 {\it{Mathematics Subject Classification}}. Primary 18E10. Secondary 18G25.\\
Key words and phrases. Path Categories, Functor Categories,  Quasi-hereditary Categories, Matrix Categories.}
\maketitle

\date{}
\begin{abstract}

In this paper, we prove that the lower triangular matrix category $\Lambda =\left  [ \begin{smallmatrix} \mathcal{T}&0\\ M&\mathcal{U} \end{smallmatrix} \right ]$, where $\T$ and $\U$ are quasi-hereditary $\mathrm{Hom}$-finite Krull-Schmidt $K$-categories and $M$ is a $\mathcal U\otimes_K \mathcal T^{op}$-module that satisfies suitable conditions, is quasi-hereditary in the sense of \cite{LGOS1} and \cite{Martin}. Moreover, we solve the problem of finding quotients of path categories isomorphic to the lower triangular matrix category $\Lambda$, where $\mathcal T=K\mathcal{R/J}$ and $\mathcal U=K\mathcal{Q/I}$ are path categories of infinity quivers modulo admissible ideals. Finally, we study the case  where $\Lambda$ is a path category of a quiver $Q$ with relations and $\mathcal U$ is the full additive  subcategory of $\Lambda$ obtained by deleting a source vertex $*$ in $Q$ and $\mathcal T=\mathrm{add} \{*\}$. We then show that there exists an adjoint pair of functors $(\mathcal R, \mathcal E)$ between the functor categories $\mathrm{mod} \ \Lambda$ and $\mathrm{mod} \ \mathcal U$ that preserve orthogonality and exceptionality; see \cite{Assem1}. We then give some examples of how to extend classical tilting subcategories of $\mathcal U$-modules to classical tilting subcategories of $\Lambda$-modules.
\end{abstract}

\section*{Introduction}
\label{intro}

Rings of the form $\left [ \begin{smallmatrix} T&0 \\ M&U \end{smallmatrix}\right]$ where $T$ and $U$ are rings and $M$ is a $T-U$-bimodule have appeared often in the study of the representation theory of artin rings and algebras  \cite{Aus3, DR2, Gordon, Fossum}. Such rings are called \emph{triangular matrix rings} and appear  in the study of homomorphic images of hereditary artin algebras and of the decomposition of algebras and the direct sum of two rings. Triangular matrix rings and their homological properties have been widely studied. The so-called one-point extension is a special case of the triangular matrix algebra, and these types of algebras have been studied in several contexts; see for example \cite{Chase, Fields, Haghany}. Recently in \cite{Zhu}, B. Zhu studied the triangular matrix algebra $\Lambda$ where $T$ and $U$ are quasi-hereditary algebras, and he proved that under suitable conditions on $M$, $\Lambda$ is a quasi-hereditary algebra.

The notion of quasi-hereditary algebra and highest weight category were introduced and studied by  E. Cline, B. Parshall and L. Scott \cite{CPS1, CPS0, Scott}. Highest weight categories arise in the representation theory of Lie algebras and algebraic groups. For the setting of finite dimensional algebras, quasi-hereditary algebras were amply studied by V. Dlab and M. Ringel in \cite{Dlab0, DR2, DR1, Rin}. In addition, they introduced the set of standard modules ${}_\Lambda\Delta$ associated with an algebra $\Lambda.$ More recently, M. Ringel studied the homological properties of the category $\mathcal{F}({}_\Lambda\Delta)$ of ${}_\Lambda\Delta$-filtered $\Lambda$-modules and constructed the characteristic tilting module ${}_\Lambda T$ associated with $\mathcal{F}(_\Lambda\Delta).$    

On the other hand, the idea that additive categories are rings with several objects was developed by B. Mitchell, see [41], who showed that a substantial amount of non-commutative ring theory is still true in this generality and that familiar theorems for rings originate from the natural development of category theory. Following this approach, M. Auslander and I. Reiten worked exhaustively with the functor categories and found numerous applications to the development of the representation theory of artin algebras; see for instance \cite{Aus, Aus1,Aus2,Aus3,Aus4}, to name a few.

More recently, R. Mart\'{\i}nez-Villa and \O. Solberg, motivated by the work on functor categories by M. Auslander in \cite{Aus2, Aus}, studied the Auslander-Reiten components of finite dimensional algebras. They did so in order to establish when the category of graded functors is Noetherian \cite{MVS1, MVS2, MVS3}. Recently, R. Mart\'{\i}nez-Villa and M. Ort\'{\i}z studied tilting theory in arbitrary functor categories, in \cite{MVO1, MVO2}. They proved that most of the properties that are satisfied by a tilting module over an artin algebra also hold for functor categories.  It is worth noting that the study of Auslander-Reiten components of finite dimensional algebras is related to the study of some infinite quivers that play an important role in the representation theory of coalgebras; see \cite{Bautista}.

Following the line of the above mentioned works, M. Ort\'{\i}z introduced in \cite{Martin} the concept of quasi-hereditary category to study the Auslander-Reiten components of a finite dimensional algebra $\Lambda.$ Similarly, as the standard modules appear in the theory of quasi-hereditary algebras, the  concept of standard functors appears in this context, which resulted in a generalization of the notion of standard modules. As a consequence, a connection is obtained between highest weight categories and quasi-hereditary categories as stated by  H. Krause in \cite{Krause2}.

Finally, following Mitchell's philosophy, the concept of the triangular matrix category is introduced  in \cite{LGOS1,LGOS2}, as the analogue of the triangular matrix algebra to the context of rings with several objects, by V. Santiago, A. Leon and M. Ortiz, and they obtain some applications  to path categories given by infinite quivers, the construction of recollements and the study of functorially finite subcategories in functor categories. 

The aim of this paper is to show that infinite quiver path categories  and quasi-hereditary categories can be constructed as triangular matrix categories,  generalizing some of the results obtained by B. Zhu in \cite{Zhu}. Furthermore if, this type of matrices are of  the one-point extension type, some homological properties can be described. It is worth mentioning that recently in \cite{Marcos} similar results have been obtained in the context of standardly stratified lower triangular $\mathbb K$-algebras with enough idempotents.

We outline the content of the paper section-by-section as follows.

In section 1, we recall basic results about path categories, functor categories, quasi-hereditary categories and triangular matrix categories.

In section 2, following the same arguments given by Leszczy\'nski in  \cite[Lemma  1.3.]{Leszczynski}, we prove that the tensor product of two path categories is again a path category  as defined by Ringel \cite{Rin2}. Thus we will use this result to solve the problem of finding quotients of path categories isomorphic to the lower triangular matrix category $ \left( \begin{smallmatrix} K\mathcal R/\mathcal J&0 \\ M&K \mathcal Q/\mathcal I \end{smallmatrix}\right)$, where $K\mathcal R/\mathcal J$ and $K\mathcal Q/\mathcal I$ are path categories modulo admissible ideals and $M$ is a functor from $K\mathcal Q/\mathcal I\otimes( K\mathcal R/\mathcal J)^{op}$ to the category  $\mathrm{mod}\ K$; see \cite{Oystein}.

In section 3, we  generalize a result given in  \cite[Theorem 3.1]{Zhu}. Specifically, we prove that if $\mathcal U$ and $\mathcal T$ are $\mathrm{Hom}$-finite Krull-Schmidt quasi-hereditary categories with respect to filtrations $\{\mathcal U_j\}_{0\le j\le n}$ and $\{\mathcal T_j\}_{j\ge 0}$ of $\mathcal U$ and $\mathcal T$, respectively, consisting of additively  closed subcategories and $M_T\in\mathcal F(_{\mathcal U}\Delta)$ for all $T\in\mathcal T$, then $\Lambda= \left(\begin{smallmatrix} \mathcal T & 0 \\ M & \mathcal{U} \end{smallmatrix}\right) $ is quasi-hereditary with respect to a certain filtration  $\{\Lambda_j\}_{j\ge 0 }$. Moreover, we obtain a characterization of the category  of the $_\Lambda \Delta$-filtered $\Lambda$-modules.

In section 4, we consider a strongly locally finite quiver $Q=(Q_0,Q_1)$ with relations $I$. Then, given a source $*$ in $Q$ and by setting the full additive subcategories of $\mathcal C:= K\mathcal Q/\mathcal I$, $\mathcal U=\mathrm{add} \{Q_0-\{*\}\}$ and $\mathcal T:=\mathrm{add}\{*\}$, we see that $\mathcal C$ is isomorphic to a triangular matrix category $\Lambda$, and we  prove that there exist an adjoint pair of  additive functors $\mathcal R:\mathrm{mod}\ \Lambda\rightarrow \mathrm{mod}\ \mathcal U$ and $\mathcal E:\mathrm{mod}\ \mathcal U\rightarrow \mathrm{mod}\ \Lambda$  which preserve orthogonality and exceptionality, generalizing some results given in \cite{Assem1}. Further on we apply this result to extend classical tilting categories to functor categories.


\section{Preliminaries}
\bigskip
\subsection{$K$-categories, Path Categories,  Representations and Functor Categories} In this part, we recall some basic definitions to approach this work.  The reader can consult  \cite{Assem} and \cite{Barot} for more details. 
 
 {\sc $K$-Categories } 
 Let  $K$ be a field. A category $\mathcal C$ is a $K$-category if for each pair  of objects $X$ and $Y$ in $\mathcal C$, the set of morphisms $\mathcal{C}(X,Y)$ is equipped with a $K$-vector space structure such that the composition $\circ$ of morphisms in $\mathcal{C}$ is a
$K$-bilinear map.  A $K$-category $\mathcal C$ is called $\mathrm{Hom}$-\textbf{finite} if  $\mathrm{dim}_K\mathcal{C}(X,Y)<\infty$.

{\sc Krull-Schmidt categories.}  A Krull-Schmidt category  is an additive category such that each object decomposes into a finite direct sum of indecomposable objects with local endomorphism rings.

{\sc Ideals.} Let $\mathcal C$ be an additive $K$-category. A class $I$  of morphisms of $\mathcal C$ is a \textbf{two-sided ideal} in $\C$ if:
(a) the zero morphism $0_X\in \mathcal C(X,X)$ belongs to $I$;
(b) if $f,g:X\rightarrow Y$  are morphisms in $I$ and $\lambda$, $\mu\in K$, then $\lambda f+ \mu g\in I$;
(c) if $f \in I$  and $g$ is a morphism in $\mathcal C$  that is left-composable with $f$, then $g\circ f\in I$ and 
(d) if $f \in I$  and $h$ is a morphism in $\mathcal C$ that is right-composable with $f$, then $ f\circ h\in I$. 
Equivalently, a two-sided ideal $I$ of $\mathcal C$  can be considered as a subfunctor 
$I (-,-) \subseteq \mathcal C (-,- ) : \mathcal C^{op} \times \mathcal C\rightarrow \mathrm{Mod}\ K $, defined by assigning to each pair $(X,Y)$ of objects 
$X$,$Y$ of $\mathcal C$ a  $K$-subspace $I(X,Y )$ of $\mathcal C(X,Y ) $ such that:
(i) if $f \in I(X,Y )$ and $g\in\mathcal C(Y,Z)$, then $gf \in I(X,Z)$; and
(ii) if $f \in  I(X,Y )$ and $h\in\mathcal C(U,X)$,  then $fh \in I(U,Z)$.

Given a two-sided ideal $I$ in an additive $K$-category $\mathcal C$, the \textbf{quotient category} $\mathcal C/I$ that will be the category with objects that are objects are the objects of $\mathcal C$ and the space of morphisms from $X$ to $Y$ in $\mathcal C/I$ is the quotient space $(\mathcal {C}/I)(X,Y) = \mathcal C(X,Y)/I(X,Y)$ of $\mathcal C(X,Y)$. It is easy to see that the quotient category $ \mathcal C/I$ is an additive $K$-category, and the projection functor $ \pi : \mathcal C\rightarrow \mathcal C/I$ assigning to each $f : X \rightarrow Y$  in $C$ the coset $f + I \in \mathcal{C}/I (X, Y )$ is a $K$-linear functor. Moreover, $\pi $ is full and dense, and $\mathrm{Ker} (\pi) =I$.

 The \textbf{ (Jacobson) radical} of an additive $K$-category $\mathcal C$  is the two-sided ideal $\mathrm{rad}_{\mathcal C}$ in $\mathcal C$ defined by the formula $\mathrm{rad}_{\mathcal C}(X,Y)=\{h\in\mathcal C(X,Y):1_X-gh \text{ is invertible for any } g\in\mathcal C(Y,X)\}$ for all objects $X$ and $Y$ of $\mathcal C$.

\begin{rk}
Let $I$ be an ideal  in a $K$-category $\mathcal C$. We clearly see then that $f=(f_{ij}): \oplus_{i=1}^n X_i\rightarrow  \oplus_{j=1}^m Y_j$ lies in $ I(\oplus_{i=1}^nX_i,  \oplus_{j=1}^mY_j)$ if and only if $f_{ij}$ lies in $I(X_i,Y_j)$ for all $i$ and $j$, such that $1\le i\le n$ and $1\le j\le m$ .
\end{rk}

{\sc Tensor Product of $K$-categories.}   Let $\mathcal C$ and $\mathcal C'$ be $K$-categories. The tensor product   $\mathcal C\otimes_K \mathcal C'$ is the $K$-category whose class of objects is $\mathrm{Obj } \ \mathcal C\times \mathrm{Obj }\ \mathcal C'$, where the set of morphisms  from $(p_1,q_1)$ to $(p_2,q_2)$ is the ordinary tensor product  $\mathcal C(p_1,p_2)\otimes \mathcal C'(q_1,q_2)$. The composition $$\mathcal C(p_1,p_2)\otimes \mathcal C'(q_1,q_2)\times \mathcal C(p_2,p_3)\otimes \mathcal C'(q_2,q_3)\rightarrow \mathcal C(p_1,p_3)\otimes \mathcal C'(q_1,q_3)$$ is given by the rule 
$((f_1\otimes g_1),(f_2\otimes g_2))\mapsto (f_2f_1\otimes g_2 g_1)$.  This compositions is bilinear; see \cite{Mitchell}.

{\sc Quivers, path algebras and path categories.}  A \textbf{quiver} is an oriented graph, formally denoted by a quadruple $Q=(Q_0,Q_1,s,t)$, with a set of vertices $Q_0$ and a set of arrows $Q_1$, and two maps $s,t: Q_1\rightarrow Q_0$, called \emph{source} and \emph{target}, defined by $s(a\rightarrow b)=a$ and $t(a\rightarrow b)=b$, respectively, if $\alpha:a\rightarrow b$ is an arrow in $Q_1$.

A path of length $l\ge 1$  from $a$ to $b$ in a quiver $Q$ is of the form $(a|\alpha_1,\ldots,\alpha_l|b)$ with arrows $\alpha_i$ satisfying $t(\alpha_i)=s(\alpha_{i+1})$ for all  $1\le i\le l$ and $a=s(\alpha_1)$ as well as $b=t(\alpha_l)$. In addition, for any vertex $a$ in $Q_0$, a path of length 0 from $a$ to itself is denoted by $\epsilon_a$.

Given a quiver $Q$, it is \textbf{path category} $K\mathcal Q$ is an additive category, with objects being direct sums of indecomposable  objects. The indecomposable  objects  in the path category are given by the set  $Q_0$, and given $a,b\in Q_0$, the set of maps from $a$ to $b$ is given by the $K$-vector space with basis the set of all paths from $a$ to $b$. The composition of maps is induced from the usual composition of paths:
\begin{equation}\label{pathcomp}
 (a|\alpha_1,\ldots,\alpha_l|b)(b|\beta_1,\ldots,\beta_s|c)=(a|\alpha_1,\ldots,\alpha_l\beta_1,\ldots,\beta_s|c),
\end{equation}
where $(a|\alpha_1,\ldots,\alpha_l|b)$ is a path from $a$ to $b$ and $(b|\beta_1,\ldots,\beta_s|c)$ is a path from $b$ to $c$.
  
  Similarly, the \textbf{path algebra} of $Q$ denoted by $KQ$, is the $K$-vector space with basis the set of all paths in $Q$, and the product of two paths is defined by (\ref{pathcomp}) if they are composable,  and it is zero if they are non-composable. In $KQ$, any ideal $I$ is generated by a set of paths $\{\rho_i| i\}$, that is $I=\langle \rho_i| i\ \rangle $. Let  $I$ be an ideal in $KQ$, then given a pair of finite sets of vertices $\{X_i\}_{i=1}^n$, $\{Y_j\}_{j=1}^m$ we set $\mathcal I(\oplus_{i=1}^n {X_i},  \oplus_{j=1}^m{Y_j})=\{(f_{ij})\in K\mathcal Q(\oplus_{i=1}^n {X_i}, \oplus_{j=1}^m{Y_j})|f_{ij}\in I\}$. This allows us to define an ideal $\mathcal I$ in $K\mathcal Q$, and we refer to it as the ideal generated by $I$. If $I\subset KQ$ is generated by the set of paths $\{\rho_i| i\}$, we say that $\mathcal I$ is generated by  the set $\{\rho_i| i\}$.
  
The ideal generated by all arrows is denoted by $KQ^{+}$. Note that $(KQ^{+})^n$ is the ideal generated by all paths of length $\ge n$. Given vertices $a,b\in Q_0$, a finite linear combination $\sum_{w}c_w w$ with $c_w\in K$ where $w$ are paths of lengths $\ge 2$ from $a$ to $b$ is called a \textbf{relation} on 
 $Q$.  Any ideal $I\subseteq  (K Q^{+})^2$ can be generated, as an ideal, by relations. An ideal $I\subset KQ$ is called \textbf{admisible} if it is generated by a set of relations. We then say that an ideal $\mathcal I$ in $K\mathcal Q$ is admissible if it is generated by an admissible ideal in $KQ$.
 
{\sc Representations of Quivers.}
 A representation of a quiver $Q$ is a pair $V=\left ( (V_i)_{i\in Q_0}, (V_\alpha)_{\alpha\in Q_1}\right )$, where each element of the family $\{V_i\}_{i\in Q_0}$ is a vector space and $V_\alpha: V_{s(\alpha)}\rightarrow V_{t(\alpha)}$ is a $K$-linear map. Let $V$ and $ W$ be two representations of $ Q$. A morphism from $V$ to $W$ is a family of linear maps $f=(f_i:V_i\rightarrow W_i)_{i\in Q_0}$  such that for each arrow  $\alpha: i\rightarrow j$  we have $f_jV_\alpha=W_\alpha f_i$. We denote by $\mathrm{rep}\  Q$ the abelian category that  has as objects the representations of $Q$ and as morphisms just the morphisms of representations.  Let $\rho=(a|\alpha_1,\ldots,\alpha_l|b)$ a path in $Q$, we set $V_\rho=V_{\alpha_l}\circ \cdots \circ V_{\alpha_1}$. Let $I\subset KQ$ be an ideal, we then say the representation $V$ is \textbf{bounded} by $I$ if $V_\rho=0$ for all $\rho\in I$.  The full subcategory of $\mathrm{rep}\ Q$ consisting of representations bounded by $I$ is denoted by $\mathrm{rep} \ (Q,I)$.
  
{\sc Strongly locally finite quivers. }  Let $Q$ be a quiver. For $x\in Q_0$, we denote by $x^+$ and $x^{-}$  the set of arrows starting in $x$ and the set of arrows ending in $x$, respectively. Recall that x is a \textbf{sink vertex} or a \textbf{source vertex} if  $x^{+}=\emptyset $ or $x^-=\emptyset$.  One says that $Q$ is \textbf{locally finite} if $x^+$ and $x^{-}$ are finite sets and \textbf{interval finite} if the set of paths from $x$ to $y$ is finite for any $x, y \in Q_0$. For short, we say that Q is \textbf{strongly locally finite} if it is locally finite and interval finite. In particular, $Q$ contains no oriented cycle in case it is interval finite. Note that under these conditions, if $Q$ is a strongly locally finite quiver, the path category $K\mathcal Q$ is a $\mathrm{Hom}$-finite Krull-Schmidt $K$-category; see \cite{Bautista}.
 
{\sc Functor categories.} Recall that a category $\C$ is said to be \textbf{skeletally small} if it has a small dense subcategory $\C'$, see \cite{Aus}. Let $\C$ be a $\mathrm{Hom}$-finite Krull-Schmidt and skeletally small $K-$category. The abelian category $(\mathcal C, \mathbf{Ab})$ is the category of all additive covariant functors from $\mathcal C$ to the category of abelian groups, which we will call $\mathcal C$-modules. Given two $\mathcal C$-modules $F$ and $G$, the set of morphisms $\mathrm{Hom}_{(\mathcal C, \mathbf{Ab})}(F,G)$ is denoted simply by $\mathrm{Hom}_\mathcal C(F,G)$. Following \cite{Aus, Aus1}, $(\mathcal C, \mathbf{Ab})$ is denoted by $\mathrm{Mod}\ \mathcal C$. We recall that a $\mathcal C$-module $M$ is \textbf{finitely presented} if an exact sequence $P_1\rightarrow P_0\rightarrow M\rightarrow 0$ of $\mathcal C$-modules exist where $P_0$ and $P_1$ are projective finitely generated $\mathcal C$-modules. We denote by $\mathrm{mod}\ \mathcal C$ the full subcategory of $\mathrm{Mod}\ \mathcal C$ consisting of finitely presented $\C$-modules. Let $M$ be a $\mathcal C$-module, so  each $C$ in $\mathcal C$ the abelian group $M(C)$ has  a structure as a $\mathrm{End}_\mathcal C(C)$-module and hence as a $K$-module since $\mathrm{End}_\mathcal C(C)$ is a $K$-algebra. We denote by $(\mathcal C, \mathrm{mod} \ K)$ the full category of $\mathrm{Mod}\ \mathcal C$ of all $\mathcal C$-modules  such that $M(C)$ is a finitely generated $K$-module. The category $(\mathcal C, \mathrm{mod} \ K)$ is an abelian category with the property that the inclusion $(\mathcal C, \mathrm{mod} \ K)\rightarrow \mathrm{Mod}\ \C$ is exact and contains $\mathrm{mod}\ \C$ as a full subcategory.
 
 Let $Q$ be a quiver and $I$ be an ideal $I\subset KQ$. Set $\mathcal C=K\mathcal Q/\mathcal I$. Then each representation $V=\left ( (V_i)_{i\in Q_0}, (V_\alpha)_{\alpha\in Q_1}\right )$ in $\mathrm{rep} \ (Q,I)$ defines a $\mathcal C$-module  $\tilde V$ in $(\mathcal C, \mathrm{mod} \ K)$ by setting $\tilde V(i)=V_i$ and $\tilde V(\alpha)=V_\alpha$. 
 
 In general, the functor $D: (\mathcal C, \mathrm{mod} \ K)\rightarrow (\mathcal C^{op}, \mathrm{mod} \ K)$ given by $D(M)(X)$ $=\mathrm{Hom}_K (M(X),K)$ for all $X$ in $\mathcal C$ defines a duality between  $(\mathcal C, \mathrm{mod} \ K)$ and  $(\mathcal C^{op}, \mathrm{mod} \ K)$, and we refer to it as the \textbf{standard duality}. 
 
  


\subsection{Quasi-hereditary categories and triangular matrix categories}

Assume $\mathcal C$ is a $\mathrm{Hom}$-finite Krull-Schmidt $K$-category. In order to generalize the notion of quasi-hereditary category, the notion of  heredity ideal and heredity chain is introduced in \cite{Martin}.  

{\sc Heredity Ideals.} A two-sided ideal $I$ in $\mathcal C$ is called (left) \textbf{heredity} if the following conditions hold: (i) $I^2=I$, i.e., $I$  is an idempotent ideal; (ii) $I\mathrm{rad}\ \C(-,?)I=0$, and  (iii) $I(X,-)$ is a projective finitely generated  $\C-$module for all $X\in \C$. $\mathcal C$ is called quasi-hereditary if there exist a chain $\{I_j\}_{j\in J}$, where $J$ is at most countable of two-sided ideals $0=I_0\subset I_1\subset \cdots \subset \mathcal C(-,?)$, which is exhaustive (that is, $\cup_{j\in J}I_j=\mathcal C(-,?)$), and $I_{j}/I_{j-1}$ is heredity in the quotient category $\mathcal C/I_{j-1}$. Such a chain is called a \textbf{heredity chain}.

Let $B$ be a full additive subcategory of $\mathcal C$. Given $C,C'\in\mathcal C$, we denote by $I_\mathcal B(C,C')$ the subset of $\mathcal C(C,C')$ consisting of morphisms which factor through some object in $\mathcal B$. This allows us to define the two-sided ideal $I_{\mathcal B}(-,?)$ which is an   idempotent ideal in $\mathcal C$. Moreover, if we denote by $\mathrm{Tr}_{\{\mathcal C(E,-)\}_{E\in\mathcal B}} \mathcal{C}(X,-)$, $X\in\mathcal C$, the trace of $\{\mathcal C(E,-)\}_{E\in\mathcal B} $ in $\mathcal{C}(X,-)$, we have $\mathrm{Tr}_{\{\mathcal C(E,-)\}_{E\in\mathcal B}} \mathcal{C}(X,-)=I_{\mathcal B}(X,-)$. 

{\sc Quasi-hereditary categories.} Assume we have an exhaustive filtration $\{\B_j\}_{j\geq 0}$ of $\C$ into additive full subcategories (that is, $\cup_{j\ge 0} \mathcal B_j=\mathcal C$ ). We then have an exhaustive chain of two-sided idempotent ideals:
\[
\{0\}=I_{\B_0} \subset I_{\B_1} \subset \cdots I_{\B_{j-1}} \subset I_{\B_j} \subset \cdots \subset \C(-,?).
\]
Note that  $\frac{I_{\B_j}}{I_{\B_{j-1}}}$ is an idempotent ideal in the quotient category  $\frac{\C}{I_{\B_{j-1}}}$ since $I_{\B_{j}}$ and  $I_{\B_{j-1}}$ are idempotents in  $\C$ and 
\[
\left(\frac{I_{\B_j}}{I_{\B_{j-1}}}\right)^2= \frac{I_{\B_j}}{I_{\B_{j-1}}}\frac{I_{\B_j}}{I_{\B_{j-1}}}=\frac{I_{\B_j}^2+I_{\B_{j-1}}}{I_{\B_{j-1}}}=\frac{I_{\B_j}}{I_{\B_{j-1}}}.\\
\]

The above motivates us  to introduce the principal definition in this section.

\begin{defi}
	Let  $\C$  be a $\mathrm{Hom}$-finite Krull-Schmidt  $K-$category. Assume that  $\{\B_j\}_{j\geq 0}$  is an  exhaustive filtration of  $\C$  into full additive subcategories. We say that   $\C$ is quasi-hereditary with respect to $\{\B_j\}_{j\geq 0}$ if 
		\[
		\{0\}=I_{\B_0} \subset I_{\B_1} \subset \cdots I_{\B_{j-1}} \subset I_{\B_j} \subset \cdots \subset \C(-,?)
		\]
is a heredity chain.
\end{defi}

Recall that a full  additive subcategory $\mathcal B$ of $\mathcal C$ is called additively closed if it is closed under direct summands and isomorphisms. The following result  given  in \cite{Martin} will be useful in the remainder of this work.

\begin{teo}\label{Mar1}
	Let  $\{\mathcal{B}_j\}_{j\geq 0}$ be an  exhaustive filtration of  $\mathcal{C}$ into additively closed full subcategories . Then  $\mathcal{C}$ is quasi-hereditary with respect to  $\{\mathcal{B}_j\}_{j\geq 0}$  if and only if the following conditions  hold:
	\begin{enumerate}
		\item[(i)]  $\mathrm{rad}_{\mathcal{C}}(E,E')=I_{\mathcal{B}_{j-1}}(E,E')$, for all pairs of objects $E,E'\in\mathrm{ Ind} \ \mathcal{B}_{j}-\mathrm{ Ind} \ \mathcal{B}_{j-1}$ ;
		\item[(ii)]  and for all $X \in \mathcal{C}$ and  $j\geq 1$, there exists an exact sequence 
		\[
		\xymatrix{\mathcal{C}(E_{j-1},-) \ar[r] & \mathcal{C}(E_{j},-) \ar[r] & I_{\mathcal{B}_{j}}(X,-) \ar[r] \ar[r] & 0},
		\]
		with  $E_{j} \in \mathcal{B}_{j}$ and $E_{j-1} \in \mathcal{B}_{j-1}$.
	\end{enumerate}
\end{teo}

{\sc Standard $\mathcal C$-modules.}  Let $\mathcal C$ be a quasi-hereditary category with respect to a family  of additively closed subcategories $\{\mathcal B_j\}$. Each module $_\mathcal C\Delta_E(j):= \mathcal C(E,-)/I_{\mathcal B_{j-1}}(E, -)$, $E\in\mathrm {Ind}  \mathcal B_j-\mathrm{Ind}  \mathcal B_{j-1}$ is called \textbf{standard}, and $_\mathcal C\Delta(j)$ denotes the category consisting of the standard $\C$- modules $_\mathcal C\Delta_E(j)$. In addition, $_\mathcal C\Delta$ denotes the full subcategory   consisting of the standard $\C$- modules.
 
{\sc Filtered  $\mathcal C$-modules.}
Let $\A$ be an abelian category, and $\ltt{X}\subseteq \A.$ We denote by $\ltt{X}^{\amalg}$ the class of objects of $\A,$which are a finite direct sum of objects in $\ltt{X}.$  We say that $M\in \A$ is $\ltt{X}$-$\textbf{filtered}$ if there exists a chain $\{M_{j}\}_{j\ge 0}$ of subobjects of $M$ such that $M_{j+1}/M_{j}\in \ltt{X}^{\amalg}$ for $j\ge 0.$ In case $M=M_n$ for some $n\in\mathbb N$, we say that 
 $M$ has a finite $\ltt{X}$-filtration of length $n$. We denote by $\ltt{F}(\ltt{X})$ the class of
objects that are $\ltt{X}$-filtered and by $\ltt{F}_f(\ltt{X})$ the class of objects  that have a finite filtration. For $M \in \ltt{F}_f(\ltt{X})$, the $\ltt{X}$-length of $M$ can be defined as follows
$l_{\ltt X}(M):=\mathrm{min}\{ n\in \mathbb{N} : M \text{ has an $\ltt{X}$-filtration of length $n$}\}$.

By using the notion of $\ltt{X}$-length and induction, the following useful remark can be proven.
\begin{rk}\label{remark1}
Let $\ltt{X}$  be a class of objects in an abelian category $A$. Then, the
class $\ltt{F}_f (\ltt{X})$ is closed under extensions.
\end{rk}

 
 Given a  $\C-$module $F$, its trace filtration with respect to $\{\B_j\}_{j\geq 0}$ is given by
\[
\{0\}=F^{[0]} \subset F^{[1]} \subset F^{[2]} \subset \cdots \subset F^{[j-1]} \subset F^{[j]} \subset \cdots,
\]
where  $F^{[j]}:=Tr_{\{\C(E,-)\}_{E\in \B_j}}F$ and $F= \bigcup\limits_{j\geq 0}F^{[j]}$.

It is of interest to study the $ \C$-modules $F$ that possess a trace  filtration that satisfies $ \frac {F ^ {[j]}} {F ^ {[j-1]}} \in  {}_{\mathcal {C}}\Delta (j) ^{\amalg }$ for all $ j \geq 1 $. It then follows that these $ \C$-modules  are $ \Delta$-filtered. We denote the full subcategory of the $ \Delta$-  filtered modules  by $ \mathcal {F} (_\mathcal C\Delta) $.

The following result will be very useful in what follows.

\begin{lem}\label{lem5}
	Let $F \in \mathcal{F}(\Delta)$.
	\begin{enumerate}
		\item[(a)] For all  $j\geq 0$, $F^{[j]}$ has a presentation
		\begin{equation}\label{eq6}
		\C(E_{j-1},-) \rightarrow \C(E_{j},-) \rightarrow  F^{[j]} \rightarrow  0,\ E_{j-1} \in \B_{j-1}, E_{j}\in \B_{j}.
		\end{equation} 
		\item[(b)]  $F^{[j]}\cong \C \otimes_{\B_j}(F|_{\B_{j}})$; see \cite[Proposition A.3.2]{Dlab0}.
	\end{enumerate}
\end{lem}

{\sc Triangular Matrix Categories.}
Following Mitchell's philosophy \cite{Mitchell}, in \cite{LGOS1, LGOS2} the notion of triangular matrix categories is introduced in order to define the analogous of the triangular matrix algebras to the context of rings with several objects.
 \begin{defi} Given  $\mathcal T$ and $\mathcal U$ additive $K$-categories and a functor 
$M:\mathcal U\otimes_K \mathcal T^{op}\rightarrow \mathrm{Mod}\ K$, the triangular matrix category  $\Lambda=
\left (\begin{smallmatrix}
	\mathcal T& 0 \\ 
	M & \mathcal U
	\end{smallmatrix}\right)$
	is the additive $K$-category whose collection of objects are  the \emph{matrices}
	$\left(\begin{smallmatrix}
	 T& 0 \\ 
	M &  U
	\end{smallmatrix}\right )$, where $U$ and $T$ are objects in $\mathcal U$ and $\mathcal T$, respectively.

Let $X=\left(\begin{smallmatrix}
	 T& 0 \\ 
	M &  U
	\end{smallmatrix}\right )$, $X'=\left(\begin{smallmatrix}
	 T'& 0 \\ 
	M &  U'
	\end{smallmatrix}\right )$ and $X''=\left(\begin{smallmatrix}
	 T''& 0 \\ 
	M &  U''
	\end{smallmatrix}\right )$ be objects  in $\Lambda$. The set of morphisms from $X$ to $X'$ is
	$\Lambda(X,X')=\left ( \begin{smallmatrix}
	\mathcal T(T,T')& 0 \\ 
	M ((U',T))& \mathcal U(U,U')
	\end{smallmatrix}\right )$, where $\left ( \begin{smallmatrix}
	\mathcal T(T,T')& 0 \\ 
	M ((U',T))& \mathcal U(U,U')
	\end{smallmatrix}\right ):=\{\left ( \begin{smallmatrix}
	 f& 0 \\ 
	m &  g\end{smallmatrix}\right )$:$ f\in T(T,T'), g\in\mathcal U(U,U'), m\in M(T,U')\}$, and the composition 
	$\Lambda(X',X'')\times \Lambda (X,X')\rightarrow \Lambda(X,X'')$ is defined by
	\[
	\hspace{1.2cm}\left(\left(\begin{smallmatrix} f_2 & 0 \\ m_2 & g_2 \end{smallmatrix}\right), \left(\begin{smallmatrix} f_1 & 0 \\ m_1 & g_1
	\end{smallmatrix}\right)\right) \mapsto \left(\begin{smallmatrix}
	f_2 \circ f_1 & 0 \\ m_2 \bullet f_1+g_2 \bullet m_1 & g_2 \circ g_1
	\end{smallmatrix}\right)
	\]
	with $m_2 \bullet f_1:=M(1_{U''}\otimes f_1^{op})(m_2)$ and $g_2 \bullet m_1:=M(g_2\otimes1_{T})(m_1)$, where $M(1_{U''}\otimes f_1^{op}):M(U'',T')\rightarrow M(U'',T)$ and
	$M(g_2\otimes 1_T):M(U'\otimes T)\rightarrow M(U''\otimes T)$ are  morphism
in $\mathrm{Mod}\ K$.
\end{defi}

\section{Tensor product of path categories and triangular categories over path categories}

\bigskip 
In \cite[Lemma 1.3]{Leszczynski},  Z. Leszczy\'nski  proves that the tensor product of two path algebras is again a path algebra. Following the same arguments given by Leszczy\'nski, we see that the result can be extended to the context of path categories as defined by Ringel \cite{Rin2}. We give the proof of this result for the benefit of the reader. We will then use these results  in the  rest of this section to solve the problem of finding quotients of path categories isomorphic to the lower triangular matrix category 
$ \Lambda = \left(
\begin{smallmatrix}
K\mathcal R/\mathcal J&0\\
M&K \mathcal Q/\mathcal I
\end{smallmatrix}
 \right)$,
where $K\mathcal R/\mathcal J$ and $K\mathcal Q/\mathcal I$ are path categories modulo admissible ideals and $M$ is a functor from $K\mathcal Q/\mathcal I\otimes( K\mathcal R/\mathcal J)^{op}$ to the category  $\mathrm{mod}\ K$; see \cite{Oystein}.

\bigskip We start with the definition of a product of two quivers.

\begin{rk}
So all the  path categories mentioned in this work   satisfy the conditions of being $\mathrm{Hom}$-finite Krull-Schmidt $K$-categories, all the quivers to which we refer will be considered 
 strongly locally finite quivers.
\end{rk}

\begin{defi}
Given two quivers $Q=(Q_0,Q_1)$ and  $Q'=(Q_0',Q_1')$,   the  product  quiver  $(Q\times Q', (Q\times Q')_0, (Q\times Q')_1)$  is defined by
$$(Q\times Q')_0=Q_0\times Q_0'\text{ and }  (Q\times Q')_1=(Q_0\times Q_1')\cup (Q_1\times Q_0').$$
\end{defi}

Let $Q$ and $Q'$ be quivers, and consider their path algebras $KQ$ and $KQ'$. Assume that $I$  and $I'$ are admissible ideals in $KQ$ and  $KQ'$, respectively.   Let  $I\square I'$ be the ideal  in $K(Q\times Q')$ generated by $(Q_0\times I')\cup (I\times Q_0')$   and the set of  relations
 \begin{equation}\label{cuadro}
(p',\beta)(\alpha,q)-(\alpha,q')(p,\beta), \ \ \ \alpha\in Q_1, \beta\in  Q_1'
\end{equation}
\[
\begin{tikzcd}[cramped, row sep=2.6em, column sep=2.8em]
	(p,q) \arrow[r,"\tiny{(\alpha,q)}"] \arrow[d,"\tiny{(p,\beta)}"] & (p',q) \arrow[d,"\tiny{(p',\beta)}"] \\
	(p,q') \arrow[r,"\tiny{(\alpha,q')}"] & (p',q').
\end{tikzcd}
\]
Similarly,   denote by $K(\mathcal Q\times\mathcal Q')$ the path category of  $Q\times Q'$  and by  $\mathcal I\square \mathcal I'$ is the ideal in  $K(\mathcal Q\times\mathcal Q')$ generated by the ideal 
$I\square I'$. 
\begin{pro}
There exists an isomorphism of $K$-categories 
\[
\frac{K(\mathcal Q\times \mathcal Q')}{\mathcal I\ \square \  \mathcal  I'}\cong \frac{K\mathcal Q}{\mathcal I}\otimes_K \frac{K\mathcal Q'}{\mathcal I'}
\]
\end{pro}

\begin{proof}
First we construct the isomorphism in the case $I=0$, $I'=0$. Let $\mathcal I_C$ be the  ideal in $K(\mathcal Q\times \mathcal Q')$ generated  by the elements (\ref{cuadro}).  Set $\mathcal B=K(\mathcal Q\times\mathcal Q')$ and $\mathcal C= K\mathcal Q$, $\mathcal C'= K\mathcal Q'$. Let $(p,r),(q,s)\in (Q\times Q')_0$  consider the map
\begin{eqnarray*}
\underline{F'}& :& \mathcal B((p,r),(q,s))\rightarrow \mathcal C(p,q)\otimes_K \mathcal C'(r,s),\\
                  &&((p,r)|\gamma_1,\ldots, \gamma_l|(q,s))\mapsto \underline F'(\gamma_l)\cdots\ \underline F'(\gamma_1),
\end{eqnarray*}
by setting   $\underline F'(\alpha,q)=\alpha\otimes\epsilon_q$ if
$\gamma=(\alpha, q)\in Q_1\times Q_0'$; 
$\underline F'(p,\beta)=\epsilon_p \otimes\beta$, if
$\gamma=(p,\beta)\in Q_0\times Q_1'$ and 
$\underline{F}'((p,q))=\epsilon_p\otimes\epsilon_q$ if $\gamma=(p,q)  \in Q_0\times Q_0'$. Extending 
by linearity the map $\underline F'$ we get a full and dense functor
$\underline F: \mathcal B\rightarrow \mathcal C\otimes_K\mathcal C'$. Since $\mathcal I_C((p,r),(q,s))\subset \mathrm{Ker} (\underline F)$, there exists an epimorphism $F: \frac{\mathcal{ B}}{\mathcal I_C}((p,r),(q,s))\rightarrow C(p,q)\otimes_K \mathcal C(r,s)$ given by 
$ F(w+\mathcal I_C)=\underline {F}(w)$ for all path  $w$ from $(p,r)$ to $(q,s)$. The above allows us to obtain a full and dense functor $F:\frac{ \mathcal B}{\mathcal I _C}\rightarrow  C\otimes_K\mathcal C'$.

\bigskip

Consider the canonical functor 
 $\mathcal C \times C'\rightarrow \mathcal C \otimes_K C'$. 
 We define a functor $\underline{H}:\mathcal C\times\mathcal  C'\rightarrow\frac{ \mathcal B}{\mathcal I _C}$ by $\underline{H}\left ((p,r)\right )=(p,r)$
 and 
 \begin{eqnarray}\label{Hundeline}
 \underline{H}&:&\mathcal C(p,q)\times\mathcal C'(r,s)\rightarrow\frac{ \mathcal B}{\mathcal I _C}((p,r),(q,s));\\
                     &&(f,g)\mapsto((q,g) \circ (f,r))+\mathcal I_C((p,r),(q,s)).\notag
 \end{eqnarray}
   It follows from the  relations   (\ref{cuadro})  that  $\underline H$ preserves compositions of morphisms; moreover, $\underline H$ is full.  Thus,  by the universal property of tensor product, $\underline{H}$ induces a full functor
  $H: \mathcal C\otimes_K\mathcal C'\rightarrow  \frac{ \mathcal B}{\mathcal I _C}$  defined by 
  $H((p,r))=(p,r)$ and $H(f\otimes g)=\underline H(f,g)$. Observe that $H$  is quasi-inverse to $F$.
  
\bigskip

Let $(f,g)\in \mathcal C(p,q)\otimes\mathcal C'(r,s)$.  First, assume that $(f,g)\in \mathcal I(p,q)\otimes\mathcal C'(r,s)$. Then, by (\ref{Hundeline}), $H(f\otimes g)\in (\frac{\mathcal I\square \mathcal I'}{\mathcal I_C})((p,r),(q,s))$ since  $(f,r)$ is in the ideal of $\mathcal B$ generated by  $I\times Q_0'$, which is contained in $\mathcal I\square \mathcal I'$.  The same occurs if  $(f,g)\in \mathcal C(p,q)\otimes\mathcal I'(r,s)$.   It follows that $$H \left( (  \mathcal I\otimes\mathcal  C'+\mathcal C\otimes\mathcal I' )\  ((p,q),(r,s))\right )\subseteq \left(\frac{\mathcal I\square \mathcal I'}{\mathcal I_C}\right)((p,r),(q,s)).$$

It is clear that $F((\mathcal I\square \mathcal I'/\mathcal I_C)((p,r),(q,s)))\subset (  \mathcal I\otimes\mathcal  C'+\mathcal C\otimes\mathcal I')\ ((p,q),(r,s))$ since $\underline{F}((\mathcal I\square \mathcal I')((p,r),(q,s)))\subset (\mathcal I\otimes\mathcal  C'+\mathcal C\otimes\mathcal I') \ ((p,q),(r,s))$.
\bigskip

On the other hand, the quotient functors $\mathcal C\rightarrow \mathcal C/\mathcal I$,  $\mathcal C'\rightarrow \mathcal C'/\mathcal I'$ induce a full functor $\pi: \mathcal C\otimes_K \mathcal C'\rightarrow \mathcal C/\mathcal I\otimes_K\mathcal C'/\mathcal I'$, with $\mathrm{Ker}(\pi)=\mathcal I\otimes\mathcal  C'+\mathcal C\otimes\mathcal I'$. Moreover, for each pair $(p,r), (q,s)\in Q_0\times  Q_0'$ the quotient functor $\pi_{\mathcal B}: \mathcal B/\mathcal I_C\rightarrow\frac{  \mathcal B/\mathcal I_C}{\mathcal I\square \mathcal I'/\mathcal I_C}\cong \frac{  \mathcal B}{\mathcal I\square \mathcal I'}$ induces an isomorphism $\tilde{F} :  \mathcal C/\mathcal I  ((p,q))\otimes_K\mathcal C'/\mathcal I'  ((r,s))\rightarrow  \frac{ \mathcal B}{\mathcal I \square \mathcal I}((p,r),(q,s))$ for which the following diagram commutes:  
\[
\begin{tikzcd}[cramped, row sep=2.6em, column sep=2.8em]
	\mathcal{C}(p,q)\otimes_K \mathcal{C}'(r,s) \arrow[r,"\pi"] \arrow[d, shift left=.6ex, "H"] & \frac{\mathcal{C}}{\mathcal I}(p,q)\otimes_K \frac{\mathcal{C}'}{\mathcal I'}(r,s) \arrow[d,"\widetilde{F}"] \\
	\frac{\mathcal{B}}{\mathcal I_{C}}((p,r),(q,s)) \arrow[u, shift left=.8ex, "F"] \arrow[r,"\pi_{\mathcal{B}}"] & \frac{\mathcal{B}}{\mathcal I\square\mathcal  I'}((p,r),(q,s)).
\end{tikzcd}
\]

Thus, we have a faithful functor $\tilde F:  K\mathcal Q/\mathcal I \otimes_K K\mathcal Q'/\mathcal I'  \rightarrow \frac{  K(\mathcal Q\times\mathcal  Q')}{\mathcal I \square \mathcal I}$, which induces an isomorphism of categories.
\end{proof}

\begin{defi}
	 Let $Q$, $R$ be disjoint quivers. Consider their path  $K$-categories   ${K\mathcal Q}$ and
	 ${K\mathcal R}$ with two-sided ideals $\mathcal I \subset  K\mathcal Q$ and $\mathcal J \subset K\mathcal R$ generated by the sets of relations $\{\rho_u\}_{u\in \mathcal{U}}$ and $\{\sigma_v\}_{v\in \mathcal{V}}$.  Let
	 $M:{K\mathcal Q}/{\mathcal I}\otimes_K ({K\mathcal R}/{\mathcal J})^{op}\rightarrow \mathrm{mod} \ K$ be a functor as described below.
	\begin{itemize}
		\item[(1)] Define $\eta(-,-):Q_0\times R_0 \longrightarrow \mathbb{N}\cup \{0\}$ by   $\eta(i,j):=\mathrm{dim}_KM(i,j)$.
		\item[(2)]  For each pair $(i,j)\in Q_0\times R_0$, let $B(i,j)=\{ b_t^{(i,j)}\}_{0\le t\le \eta(i,j)}$ be a $K$-basis  for $M(i,j)$.  
		\item[(3)]  For each pair $(i,j)\in Q_0\times R_0$, let 
		 $\overrightarrow{B}(j,i)=\{\overrightarrow{b}_t^{(j,i)}\}_{0\le t \le \eta(i,j)}$ be a set  of arrows $\overrightarrow{b}_t^{(j,i)}:j \rightarrow i$.
		\item[(4)] Define \[\mathcal{B}:=\underset{(i,j)\in Q_0 \times R_0}{\cup}B(i,j) \  \  \text{ and   } \ \ \overrightarrow{\mathcal{B}}:=\underset{(j,i)\in R_0 \times Q_0}{\cup}\overrightarrow{B}(j,i).\] 
		 \item[(5)] The augmented quiver of $R$ and $Q$ by $\mathcal{B}$, which is denoted by $(R,\mathcal{B},Q)$, is the quiver with the following set of vertices and arrows :
	\begin{eqnarray*}
	(R,\mathcal{B},Q)_0 &:=& R_0 \cup Q_0,\\
(R,\mathcal{B},Q)_1 &:=& R_1 \cup Q_1 \cup \overrightarrow{\mathcal{B}}. 
	\end{eqnarray*}
	\end{itemize}	
		\end{defi}
		
\begin{rk}	\label{remarkquiver1}

\begin{itemize}
\item[(i)] Since $Q$  is a full subquiver of $(R, \mathcal{B}, Q)$, any relation on $Q$  can be seen as a relation on $(R, \mathcal{B}, Q)$, and  $K\mathcal Q$  is a full  subcategory of $K(\mathcal R, \mathcal{B}, \mathcal Q)$. The same for the subquiver  $R\subset (R, \mathcal{B}, Q).$
\item[(ii)]  We denote by $\mathcal I\cup \mathcal J $  the ideal of $K(\mathcal R, \mathcal{B}, \mathcal Q)$ generated by  the set of relations   $\{\rho_u\}_{u\in \mathcal{U}}\cup\{\sigma_v\}_{v\in \mathcal{V}}$ seen as relations on  $(R, \mathcal{B}, Q)$. 
 \item[(iii)] Since  $Q$ and $R$ are disjoint quivers, we naturally have additive fully faithful functors $G: \frac{K\mathcal Q}{\mathcal I}\rightarrow \frac{K(\mathcal R,\mathcal B, \mathcal Q)}{\mathcal I\cup \mathcal J}$  and $H: \frac{K\mathcal R}{\mathcal J}\rightarrow \frac{K(\mathcal R,\mathcal B, \mathcal Q)}{\mathcal I\cup \mathcal J}$. The functor $G$ is defined in objects by the inclusion map, $G(i)=i$, for all $i\in  Q_0$ and in morphisms by  $ G: \mathrm{Hom}_{\frac{K\mathcal Q}{\mathcal I}}(i,i') \longrightarrow \mathrm{Hom}_{\frac{K(\mathcal R,\mathcal{B},\mathcal Q)}{\mathcal I\cup \mathcal J}}(i,i')$ by  $[f]\mapsto [[f]]$, for all pair $i, i'\in Q_0$.  
 
  If there is no risk of confusion, we simply write $[f]$ instead $[[f]]$ to refer to the class of $f$ 
in  $\mathrm{Hom}_{\frac{K(\mathcal R,\mathcal{B},\mathcal Q)}{\mathcal I\cup \mathcal J}}(i,i')$. Analogously, the functor $H$ is defined.
 \end{itemize}
 \end{rk}	
\bigskip

In the rest of this section, we use the notation given in Remark \ref {remarkquiver1}.  Let $(i,j)\in Q_0\times R_0$. Define  a $K$-linear map, which is an isomorphism:
\[
\overrightarrow{(-)}:M(i,j) \xrightarrow{\cong} \mathrm{Hom}_{\frac{K(\mathcal R,\mathcal{B},\mathcal Q)}{\mathcal I\cup \mathcal J}}(j,i)
\]
induced by the map  $B(i,j) \rightarrow \mathrm{Hom}_{\frac{K(\mathcal R,\mathcal{B},\mathcal Q)}{\mathcal I\cup \mathcal J}}(j,i), \ b_t^{(i,j)} \hspace{.08cm} \mapsto \hspace{.2cm} \overrightarrow{b}_t^{(j,i)},  \ \  0\le t\le \eta(i,j).$

\begin{lem}\label{onlyAbelian}
Let $X_0=\left (\begin{smallmatrix}
j_0&0\\
M& i_0
\end{smallmatrix}\right )$, $X_1= \left(\begin{smallmatrix}
j_1&0\\
M& i_1
\end{smallmatrix}\right )\in\Lambda$, with  $(i_0,j_0), (i_1,j_1)\in Q_0\times R_0$. The map
$\mathrm{Hom}_\Lambda(X_0,X_1)\rightarrow  \mathrm{Hom}_{\frac{K(\mathcal R,\mathcal{B},\mathcal Q)}{\mathcal I\cup \mathcal J}}(j_0\oplus i_0,j_1\oplus i_1)$ given by
\begin{eqnarray*}
\begin{pmatrix}
Hom_{\frac{K\mathcal R}{\mathcal I}}(j_0,j_1)	&0\\
M(i_1,j_0)& Hom_{\frac{K\mathcal Q}{\mathcal J}}(i_0,i_1)
\end{pmatrix}\ni	  \begin{pmatrix}
[r]&0\\
m& [q]
\end{pmatrix}
\mapsto \begin{pmatrix}
H([r])&0\\
\overrightarrow{m}&G([q]) 
\end{pmatrix}
\end{eqnarray*}
is a morphism of abelian groups. 
\end{lem}
\begin{proof}
Straightforward.
\end{proof}

Consider the the following set of relations in  $\frac{K(\mathcal R,\mathcal{B},\mathcal Q)}{\mathcal I\cup \mathcal J}$:
\begin{eqnarray*}
\mu(R,\mathcal{B},Q)&:=&\{ [q]\overrightarrow{b}-\overrightarrow{q\bullet b} \hspace{.1cm} | \hspace{.1cm} b \in B(i,j), q\in 
Q_1,   i=s(q),  \}_{ (i,j)\in Q_0\times R_0}\\
&& \cup \{ \overrightarrow{b}[r]-\overrightarrow{b\bullet r} \hspace{.1cm} | \hspace{.1cm} b \in B(i,j), r\in R_1,  j=t(r)  \}_{ (i,j)\in Q_0\times R_0},
\end{eqnarray*}
where $t(r)$ and $s(q)$ denote the target of $r$ and the source of $q$, respectively.


\begin{rk}
\begin{itemize}
\item[(i)]  Note that  $[q]\overrightarrow{b}$ and $\overrightarrow{b}[r]$ are compositions of paths in  $(R,\mathcal{B},Q)$. On the other hand, $\overrightarrow{q\bullet b}$  and  $\overrightarrow{b\bullet r}$ are linear combinations of arrows taken from $\overrightarrow{B}(j,i')$ and $\overrightarrow{B}(j',i)$, respectively. For example, if $ b \in B(i,j), [r]\in \frac{K\mathcal R}{\mathcal J}(j',j)$,  $\overrightarrow{b\bullet r}$ is obtained through the following composition:
$$ M(i,j) \xrightarrow{M(1_i\otimes r^{op})} M(i,j') \xrightarrow{\overrightarrow{(-)}} \mathrm{Hom}_{\frac{K(\mathcal R,\mathcal{B},\mathcal Q)}{\mathcal I\cup \mathcal J}}(j',i),$$
that is $ \overrightarrow{M(1_i\otimes r^{op})(b)}=\overrightarrow{b\bullet r} $.


\item[(ii)] Set $\mu=\langle\mu(R,\mathcal{B},Q)\rangle$. We have a projection
\[
\pi: \frac{K(\mathcal R,\mathcal{B},\mathcal Q)}{\mathcal I\cup \mathcal J} \longrightarrow \frac{K(\mathcal R,\mathcal{B},\mathcal Q)}{\mathcal I\cup \mathcal J \cup \mu} .
\]
Moreover, $\pi \circ G$ and  $\pi \circ H$  are fully faithful functors.
	\end{itemize}
\end{rk}

\begin{teo}
With the above notation, there exists an isomorphism of categories
\[
F: \Lambda = \left(\begin{matrix} \frac{K\mathcal R}{\mathcal J} & 0 \\ M & \frac{K\mathcal Q}{\mathcal I}
\end{matrix}\right) \longrightarrow \frac{K(\mathcal R,\mathcal{B},\mathcal Q)}{\mathcal I\cup \mathcal J \cup \mu}.
\]
\end{teo}

\begin{proof}
Let $X_0=\left ( \begin{smallmatrix}
j_0&0\\
M& i_0
\end{smallmatrix}\right )$, $X_1= \left (\begin{smallmatrix}
j_1&0\\
M& i_1
\end{smallmatrix}\right )$, $X_2= \left (\begin{smallmatrix}
j_2&0\\
M& i_2
\end{smallmatrix}\right )\in\Lambda$ with  $(i_0,j_0)$, $(i_1,j_1)$, $(i_2,j_2)\in Q_0\times R_0$. The map
$F: \mathrm{Hom}_\Lambda(X_0,X_1)\rightarrow  \mathrm{Hom}_{\frac{K(\mathcal R,\mathcal{B},\mathcal Q)}{\mathcal I\cup \mathcal J}}(j_0\oplus i_0,j_1\oplus i_1)$ is defined by composition of the following morphisms of abelian groups:
\[
\begin{diagram}
\node{ \mathrm{Hom}_\Lambda(X_0,X_1)}\arrow{e,t}{} \arrow{se,t}{F}
 \node{\mathrm{Hom}_{\frac{K(\mathcal R,\mathcal{B},\mathcal Q)}{\mathcal I\cup \mathcal J}}(j_0\oplus i_0,j_1\oplus i_1)}\arrow{s,r}{\pi}\\
\node{} \node{\mathrm{Hom}_{\frac{K(\mathcal R,\mathcal{B},\mathcal Q)}{\mathcal I\cup \mathcal J\cup\theta}}(j_0\oplus i_0,j_1\oplus i_1),}
\end{diagram}
\]
where the horizontal morphism is given in  Lemma \ref{onlyAbelian}.  

Let $ \left[\begin{smallmatrix} [r] & 0 \\ m & [q] \end{smallmatrix}\right]\in$ $\mathrm{Hom}_\Lambda(X_0,X_1)$, then $F\left (     \left[\begin{smallmatrix} [r] & 0 \\ m & [q] \end{smallmatrix}\right] \right ) =\left[\begin{smallmatrix}\pi([r]) & 0 \\ \pi(\overrightarrow{m}) &\pi ([q]) \end{smallmatrix}\right] $.  It is clear that $F$ is additive because it is a composition of additive morphisms. 

\bigskip Now we see that $F$  preserves composition. Indeed,  we have that
\begin{eqnarray*}
\mathrm{Hom}_\Lambda (X_0,X_1)&=& \begin{pmatrix} Hom_{\frac{K\mathcal R}{\mathcal J}}(j_0,j_1) & 0 \\ M(i_1,j_0) & Hom_{\frac{K\mathcal Q}{\mathcal I}}(i_0,i_1)
		\end{pmatrix},\\
\mathrm{Hom}_\Lambda (X_1,X_2)&=& \begin{pmatrix} Hom_{\frac{K\mathcal R}{\mathcal J}}(j_1,j_2) & 0 \\ M(i_2,j_1) & Hom_{\frac{K\mathcal Q}{\mathcal I}}(i_1,i_2)
		\end{pmatrix}.		
\end{eqnarray*}

Let  $\left[\begin{smallmatrix} [r_1] & 0 \\ m_1 & [q_1] \end{smallmatrix}\right]  \in  \mathrm{Hom}_\Lambda (X_0,X_1)$ and  $\left[\begin{smallmatrix} [r_2] & 0 \\ m_2 & [q_2] \end{smallmatrix}\right] \in  \mathrm{Hom}_\Lambda (X_1,X_2)$. Then
\begin{eqnarray*}
F\left(\left[\begin{matrix} [r_2] & 0 \\ m_2 & [q_2] \end{matrix}\right] \circ \left[\begin{matrix} [r_1] & 0 \\ m_1 & [q_1]
\end{matrix}\right]\right) &=&  \left[\begin{matrix}	\pi\circ H ([r_2 r_1]) & 0 \\ \pi(\overrightarrow{m_2 \bullet r_1+q_2 \bullet m_1}) & \pi\circ G([q_2 q_1]) \end{matrix}\right] \\
		&=& \left[\begin{matrix}	\pi\circ H([r_2 r_1]) & 0 \\ \pi(\overrightarrow{m_2 \bullet r_1}+\overrightarrow{q_2 \bullet m_1}) & \pi\circ G([q_2 q_1]) \end{matrix}\right].
\end{eqnarray*}
On the other hand, we have 
\begin{eqnarray*}
\small F\left(\left[\begin{matrix} [r_2] & 0 \\ m_2 & [q_2] \end{matrix}\right]\right) \circ F\left(\left[\begin{matrix} [r_1] & 0 \\ m_1 & [q_1]	\end{matrix}\right]\right) 
		&=& \left[\begin{matrix}	\pi\circ H([r_2 r_1]) & 0 \\ \pi(\overrightarrow{m_2} [r_1]+ [q_2]\overrightarrow{m_1}) & \pi\circ G([q_2 q_1]) \end{matrix}\right] .
\end{eqnarray*}
After writing $m_1$ and $m_2$  in terms of the elements  of $B(i_1,j_0)$ and $B(i_2,j_1)$, respectively, we see that $\overrightarrow{m_2} [r_1]-\overrightarrow{m_2 \bullet r_1}$ and $[q_2]\overrightarrow{m_1} -\overrightarrow{q_2\bullet m_1}$ lie in
$\mu$; therefore,  $\pi(\overrightarrow{m_2} [r_1])=\pi(\overrightarrow{m_2 \bullet r_1})$ and $\pi([q_2]\overrightarrow{m_1} )=\pi(\overrightarrow{q_2\bullet m_1})$.
 
 \bigskip
 
The functor $F$ is fully  faithful.  Let $\left[\begin{smallmatrix} [r] & 0 \\ m & [q] \end{smallmatrix}\right] \in  \mathrm{Hom}_\Lambda (X_0,X_1)$ such that $F\left(\left[\begin{smallmatrix} [r] & 0 \\ m & [q]	\end{smallmatrix}\right] \right) = \left[\begin{smallmatrix} \pi([r]) & 0 \\ \pi(\overrightarrow{m}) & \pi([q]) \end{smallmatrix}\right] = 0$. It follows that $\pi\circ H([r])=0$ and   $\pi \circ G([q])=0$ imply $[r]=0$ and $[q]=0$ since  $\pi\circ H$ and  $\pi \circ G$ are fully faithful. After writing  $m$  in terms of the basis vectors in $B(i_1,j_0)$, we have $m=\sum_{t=1}^{\eta(i_1,j_0)} \lambda_t b_t^{(i_1,j_0)}$, and $\overrightarrow{m}=\sum_{t=1}^{\eta(i_1,j_0)} \lambda_t \overrightarrow{b}_t^{(j_0,i_1)}$. Thus $\pi(\overrightarrow{m})=0$ implies that  $\overrightarrow{m}\in \mu$, and it can be written as 
 \[
		\overrightarrow{m}=\sum_{u} \alpha_u\left([q_u]\overrightarrow{b}_u - \overrightarrow{q_u \bullet b_u}\right) + \sum_{v} \beta_v\left(\overrightarrow{b}_v[r_v] - \overrightarrow{b_v \bullet r_v}\right),
		\]
		with $\alpha_u, \beta_v\in K$.  Since the left side of the above equation is a linear combination of arrows and $r:=\sum_{u} \alpha_u\left([q_u]\overrightarrow{b}_u \right )+ \sum_{v} \beta_v\left(\overrightarrow{b}_v[r_v] \right)$  is a linear combination of path of length at least two, which are linearly independent in   $\mathrm{Hom}_{\frac{K(\mathcal R,\mathcal{B},\mathcal Q)}{\mathcal I\cup \mathcal J}}(j_0\oplus i_0,j_1\oplus i_1)$    because they lie not in $I\cup J$, we must have  $r=0$ and  $\alpha_u=\beta_v=0$, which implies $\overrightarrow {m}=0$, and finally $\lambda_t=0$, for $ 1\le t \le \eta(i_1,j_0)$, since $\overrightarrow{(-) }:  M(i_1,j_0)\rightarrow   \mathrm{Hom}_{\frac{K(\mathcal R,\mathcal{B},\mathcal Q)}{\mathcal I\cup \mathcal J}}(j_0, i_1)$ is an isomorphism.
		\bigskip
			
On the other hand, since $\pi\circ F$ and $\pi\circ G$ are full functors, it is sufficient to show that for every path $\gamma:j_0 \longrightarrow i_1$ in  $\mathrm{Hom}_{\frac{K(\mathcal R,\mathcal{B},\mathcal Q)}{\mathcal I\cup \mathcal J \cup \mu}}(j_0,i_1)$ there exists an element
$m\in M(i_1,j_0)$ such that  $F\left ((\begin{smallmatrix} 0 & 0 \\ m & 0
		\end{smallmatrix}\right)) =\gamma$. But $\gamma$ can be written as 
		\[
		  \gamma=\overrightarrow{b}^{(j_1,i_1)}[r] \hspace{1cm} \text{or} \hspace{1cm} \gamma=[q]\overrightarrow{b}^{(j_0,i_0)},
		  \]
		  as is shown in the picture
		
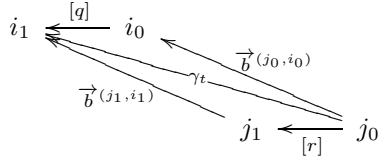
\[
		\xymatrix{i_1 \hspace{.1cm} & \hspace{.1cm} i_0 \hspace{.1cm} \ar[l]_{[q]} &  & \\	& & \hspace{.1cm} j_1 \hspace{.1cm} \ar[llu]^{\overrightarrow{b}^{(j_1,i_1)}} & \hspace{.1cm} j_0 \hspace{.1cm} .\ar[l]^{[r]} \ar[llu]_{\overrightarrow{b}^{(j_0,i_0)}} \ar[lllu]|{\gamma_t}} 
		\]
If we set $m=b^{(i_1,j_1)}\bullet r$ \hspace{.1cm} or  \hspace{.1cm} $m=q \bullet b^{(i_0,j_0)}$, we get what we desire. 		
\bigskip

		
		
		 	
		 The last assertion follows from the fact that $F$ is an isomorphism on  objects. 
\end{proof}
	
\begin{ex}\label{ex1}
Consider the quivers   
	
$$
Q: \ \begin{tikzcd}[cramped, sep=normal]
1'  & 2' \arrow[r,"\tiny{\gamma_2}"] \arrow[l,"\gamma_1"']& 3' & 4' \arrow[r,"\gamma_4"] \arrow[l,"\gamma_3"'] & 5' \arrow[r,phantom,"\cdots"] &  \hspace{.01cm}   
\end{tikzcd}
$$
and  
 \[
R:\ 
\begin{tikzcd}[cramped, sep=normal]
	1  \arrow[r, bend left, "\alpha_1", pos=0.46] & 2 \arrow[r, bend left, "\alpha_2"] \arrow[l, bend left, "\beta_1"] & 3 \arrow[r, bend left, pos=0.44, "\alpha_3"] \arrow[l, bend left, "\beta_2"] & \arrow[l, bend left, pos=0.56, "\beta_3"] & \arrow[l,phantom,"\cdots"]    
\end{tikzcd}
\]
with the set of relations 
$$J=\{\beta_1\alpha_1 \ \text{and} \ \alpha_{t+1}\alpha_t, \beta_{t}\beta_{t+1}, \alpha_t\beta_t-\beta_{t+1}\alpha_{t+1}, t\ge 1\}.$$
 Let $K\mathcal Q$ and  $K\mathcal R/\mathcal J$ be their respective path categories. Thus, by any functor $M: K\mathcal Q\otimes K\mathcal R^{op}\rightarrow \mathrm{mod}\ K$ can be identified with a functor  $M:K(\mathcal Q\times \mathcal R^{op})/0\square \mathcal J\rightarrow \mathrm{mod}\ K$, where $0\square\mathcal J$ is generated by the sets of  relations $Q_0\times \mathcal J$ and 
$\{(\gamma,t(\alpha))(s(\gamma),\alpha)-(t(\gamma),\alpha)(\gamma,s(\alpha)), (\gamma,t(\beta))(s(\gamma),\beta)-(t(\gamma),\beta)(\gamma,s(\beta)): \gamma\in Q_1, \alpha, \beta\in R_1\}$. In this example, we consider a representation on the right below  that can be seen as a functor $M: K\mathcal Q\otimes K\mathcal R^{op}\rightarrow \mathrm{mod} \ K$. \vspace{-.6cm}
\[
\small\begin{tikzcd}[cramped, row sep=3em, column sep=4.5em]
	\hspace{.01cm} & \hspace{.01cm} &  \\
	(1',3) \arrow[d, bend left=-22, "1\otimes \alpha_2^{op}"'] \arrow[u,phantom,"\vdots"] & (2',3) \arrow[d, bend left=-22, "1\otimes \alpha_2^{op}"'] \arrow[u,phantom,"\vdots"] \arrow[l,"\gamma_1\otimes1"'] \arrow[r,"\gamma_2\otimes1"] & \cdots\\
	(1',2)  \arrow[d, bend left=-22, "1\otimes \alpha_1^{op}"'] \arrow[u, bend left=-22, "1\otimes \beta_2^{op}"'] & (2',2) \arrow[d, bend left=-22, "1\otimes \alpha_1^{op}"'] \arrow[u, bend left=-22, "1\otimes \beta_2^{op}"'] \arrow[l,"\gamma_1\otimes1"'] \arrow[r,"\gamma_2\otimes1"] & \cdots\\
	(1',1) \arrow[u, bend left=-22, "1\otimes \beta_1^{op}"'] & (2',1) \arrow[u, bend left=-22, "1\otimes \beta_1^{op}"'] \arrow[l,"\gamma_1\otimes1"'] \arrow[r,"\gamma_2\otimes1"] & \cdots\\
\end{tikzcd} \hspace{.5cm}
\begin{tikzcd}[cramped, row sep=3em, column sep=3.5em]
	\hspace{.01cm} & \hspace{.01cm} & \hspace{.01cm} \\
	0 \arrow[d, bend left=-22, ""'] \arrow[u,phantom,"\vdots"] & 0 \arrow[d, bend left=-22, ""'] \arrow[u,phantom,"\vdots"] \arrow[l] \arrow[r] & 0 \arrow[u,phantom,shift left=1.4ex, "\vdots"] \arrow[d, shift right=1.4ex, bend left=-22, ""'] \cdots\\
	K \arrow[d, bend left=-22, "\tiny{\begin{pmatrix} 0 \\ 1 \end{pmatrix}}"'] \arrow[u, bend left=-22, ""'] & 0 \arrow[d, bend left=-22, ""']  \arrow[u, bend left=-22, ""'] \arrow[l] \arrow[r] & 0 \arrow[d, shift right=1.4ex, bend left=-22, ""']  \arrow[u, shift left=1.4ex, bend left=-22, ""']  \cdots\\
	K^2 \arrow[u, bend left=-22, "\tiny{\begin{pmatrix} 1 & 0 \end{pmatrix}}"'] & 0 \arrow[u, bend left=-22, ""'] \arrow[l] \arrow[r] & 0  \arrow[u, shift left=1.4ex, bend left=-22, ""']  \cdots\\
\end{tikzcd}\vspace{-.6cm}
\]
We see that $M(1',1)=K^2=\langle \varphi, \psi\rangle$, $M(1',2)=K=\langle \theta \rangle$, where
$\varphi =(1\ \ 0)$, $\psi=(0\ \ 1)$ and $\theta=1$. Thus, 
$\overrightarrow{\mathcal B}=\{\overrightarrow{\varphi}, \overrightarrow{\psi}:1\rightarrow 1', \overrightarrow{\theta}: 2\rightarrow 1'\}$ and 
\[
(R,\mathcal{B},Q):
\begin{tikzcd}[cramped, sep=normal]
	   1'  & 2' \arrow[r,"\tiny{\gamma_2}"] \arrow[l,"\gamma_1"']& 3' & 4' \arrow[r,"\gamma_4"] \arrow[l,"\gamma_3"'] & 5' \arrow[r,phantom,"\cdots"] &  \hspace{.01cm}\\
	  1 \arrow[u, shift left=.7ex ,"\varphi"] \arrow[u, shift left=-.6ex ,"\psi"'] \arrow[r, bend left, "\alpha_1", pos=0.46] & 2 \arrow[ul,"\theta"'] \arrow[r, bend left, "\alpha_2"] \arrow[l, bend left, "\beta_1"] & 3 \arrow[r, bend left, "\alpha_3"] \arrow[l, bend left, "\beta_2"] & 4 \arrow[r, bend left, "\alpha_4"] \arrow[l, bend left, "\beta_3"] & 5  \arrow[l, bend left, "\beta_4"] \arrow[r,phantom,"\cdots"] &  \hspace{.01cm} 
\end{tikzcd}
\]
On the other hand, $M(1\otimes\alpha_1^{op})=\begin{pmatrix}
0\\
1
\end{pmatrix}$, $M(1\otimes\beta_1^{op})=\begin{pmatrix}
1&0
\end{pmatrix}$ and $M(1\otimes\alpha_1^{op})(\theta)=\psi$, $M(1\otimes\beta_1^{op})(\psi)=0$, 
$M(1\otimes\beta_1^{op})(\varphi)=\theta$. Thus, 
$\mu=\{\overrightarrow{\psi}\overrightarrow{\beta}, \overrightarrow{\varphi}\overrightarrow{\beta}-\overrightarrow{\theta}, \overrightarrow{\theta}\overrightarrow{\alpha}-\overrightarrow{\psi}\}$. In this way
\[
\Lambda =  \left( \begin{smallmatrix}
 K\mathcal Q&0\\
 M&\frac{K\mathcal R}{\mathcal J}
\end{smallmatrix}\right )\cong \frac{K(\mathcal R,\mathcal{B},\mathcal Q)}{\mathcal J\cup \mathbf{\mu}}.\
\]

Although  $\mathcal J\cup \mathbf{\mu}$ is not an admissible ideal in $K(\mathcal R,\mathcal{B},\mathcal Q)$ , we can eliminate some arrows in $(R,\mathcal{B},Q)$  and obtain a quiver $Q'$  so that $\Lambda$  is isomorphic to a  path category  $K\mathcal Q'$ modulo an admissible ideal; see \cite{Oystein}.

Consider the quiver
\[
\small \begin{tikzcd}[cramped, sep=normal]
   Q': \ \ \ \arrow[r,phantom,"\cdots"] &   \arrow[r,"\gamma_3"'] & 3' & 2' \arrow[l,"\tiny{\gamma_2}"] \arrow[r,"\gamma_1"']  & 1' &
	1 \arrow[l,"\varphi"]\arrow[r, bend left, "\alpha_1", pos=0.46] & 2 \arrow[ll, bend right=50,"\theta"'] \arrow[r, bend left, "\alpha_2"] \arrow[l, bend left, "\beta_1"] & 3 \arrow[r, bend left, "\alpha_3"] \arrow[l, bend left, "\beta_2"] & \arrow[l, bend left, "\beta_3"] & \arrow[l,phantom,"\cdots"],   
\end{tikzcd}
\]
and let $\mu'=\{{\theta} {\alpha_1}{\beta_1}, {\theta}-{\varphi}{\beta_1}\}$ be a set of paths in $Q'$.

Therefore, the correspondence $[\psi]=[\theta\alpha_1]\leftrightarrow[\theta\alpha_1], [\beta]\leftrightarrow [\beta], [\alpha]\leftrightarrow [\alpha]$, $ [\gamma]\leftrightarrow [\gamma]$  induces an equivalence
$\frac{K(\mathcal R,\mathcal{B},\mathcal Q)}{\mathcal J\cup \mathbf{\mu}}\cong \frac{K\mathcal Q'}{\mu'\cup \mathcal J}$. Again if we consider the quiver,
\[
\begin{tikzcd}[cramped, sep=normal]
  Q'': \ \ \   \arrow[r,phantom,"\cdots"] &   \arrow[r,"\gamma_3"'] & 3' & 2' \arrow[l,"\tiny{\gamma_2}"] \arrow[r,"\gamma_1"']  & 1' &
	1 \arrow[l,"\varphi"]\arrow[r, bend left, "\alpha_1", pos=0.4] &  \arrow[r, bend left, "\alpha_2"] \arrow[l, bend left, "\beta_1"] & 3 \arrow[r, bend left, "\alpha_3"] \arrow[l, bend left, "\beta_2", pos=0.45] & \arrow[l, bend left, "\beta_3"] & \arrow[l,phantom,"\cdots"]   
\end{tikzcd}
\]
then the correspondence $[\theta]=[\varphi\beta_1]\leftrightarrow[\varphi\beta_1], [\beta]\leftrightarrow [\beta], [\alpha]\leftrightarrow [\alpha]$, $ [\gamma]\leftrightarrow [\gamma]$  induces an equivalence
$K\mathcal Q'/\mu'\cup \mathcal J\cong K\mathcal Q''/\mathcal J$.  In this way $\Lambda\cong K\mathcal Q''/\mathcal J$.
\end{ex}

\section{Triangular Matrix Categories over quasi-hereditary categories}

Let  $\mathcal{U}$ and $\mathcal{T}$ be $\mathrm{Hom}$-finite Krull-Schmidt quasi-hereditary $K$-categories with respect to filtrations $\{\mathcal U_j\}_{0\le j\le n}$ and  $\{\mathcal T_j\}_{j\ge 0}$, respectively, consisting of full additively closed  subcategories. Assume that  the $\mathcal U$-module $M_T=M(-,T)$ lies in $\mathcal  F(_\mathcal U\Delta)$ for all $T\in\mathcal T$; therefore $M_T$ is finitely presented  since the filtration of $\mathcal U$ is finite; see \cite[Lemma 3.18]{Martin}. Thus,  $\Lambda = \left(\begin{smallmatrix} \mathcal{T} & 0 \\ M & \mathcal{U} \end{smallmatrix}\right)$  is  a  $\mathrm{Hom}$-finite Krull-Schmidt $K$-category; see \cite[Proposition 6.9]{LGOS1}.   
\bigskip 

Consider the filtration of $\Lambda $ into  subcategories $\{\Lambda_j\}_{j\ge 0}$ given by 
\begin{eqnarray}\label{Filt1}
\Lambda_0&=&\left(\begin{smallmatrix} 0 & 0 \\ M & 0 \end{smallmatrix}\right);\notag\\
\Lambda_j&=&\left(\begin{smallmatrix} 0 & 0 \\ M & \mathcal{U}_j \end{smallmatrix}\right):=\left\{\left( \begin{smallmatrix} 0 & 0 \\ M & U \end{smallmatrix} \right): U \in \mathcal{U}_j \right\}, \text{ if $1\le j\le n$};\\
\Lambda_{n+j}&=&\left(\begin{smallmatrix} \mathcal T_{j} & 0 \\ M & \mathcal{U} \end{smallmatrix}\right)=\left\{ \left(\begin{smallmatrix} T & 0 \\ M & U \end{smallmatrix}\right) :T\in\mathcal T_j,  U \in \mathcal{U} \right\}, \text{ if $j\ge 1$}\notag.
\end{eqnarray} 

It is clear that $\Lambda_j\subseteq\Lambda $ is an additive full subcategory for all $j\ge 0$.  Moreover, if we define 
\begin{eqnarray*}
\left(\begin{smallmatrix} 0 & 0 \\ M &\mathrm{Ind}\   \mathcal{U}_j \end{smallmatrix}\right)&:=&\left\{ \left(\begin{smallmatrix} 0 & 0 \\ M & U \end{smallmatrix}\right) : U \in \mathrm{Ind} \ \mathcal{U}_j \right\}, \text{ if $1\le j\le n$, and},\\
 \left(\begin{smallmatrix} \mathrm{Ind} \ \mathcal T_{j} & 0 \\ M & 0 \end{smallmatrix}\right)&:=&\left\{ \left(\begin{smallmatrix} T & 0 \\ M & 0 \end{smallmatrix}\right) :T\in\mathcal T_j \right\}, \text{if $j\ge 1$ }.
\end{eqnarray*} 
It follows that 
\begin{eqnarray*}
\mathrm{Ind}\ \Lambda_j&=& \left(\begin{smallmatrix} 0 & 0 \\ M &\mathrm{Ind}\   \mathcal{U}_j \end{smallmatrix}\right), \text{ if $1\le j\le n$, and,}\\ \mathrm{Ind} \ \Lambda_{n+j}&=& \left(\begin{smallmatrix} 0 & 0 \\ M & \mathrm{Ind} \ \mathcal{U}=\mathrm{Ind}\ \mathcal{U}_n \end{smallmatrix} \right) \cup \left(\begin{smallmatrix} \mathrm{Ind} \ \mathcal T_{j} & 0 \\ M & 0 \end{smallmatrix}\right), \text{ if $j\ge 1$}.
\end{eqnarray*}

In this way, we have that $\Lambda_j$  , for $j\ge 0$, is additively closed. Moreover, 
\begin{eqnarray*}
\mathrm{Ind}\  \Lambda_{j } - \mathrm{Ind} \ \Lambda_{j-1}=
\begin{cases}
\left\{ \left(\begin{smallmatrix} 0 & 0 \\ M &U\end{smallmatrix}\right) : U\in\mathrm{Ind} \ \mathcal U_j-
\mathrm{Ind}\  \mathcal U_{j-1}\right \}, \text { if $1\le j\le n$}, \\
{}\\
\left \{ \left(\begin{smallmatrix} T & 0 \\ M &0\end{smallmatrix}\right) : T\in\mathrm{Ind} \ \mathcal T_{j-n}-
\mathrm{Ind}\  \mathcal T_{j-n-1}\right \}, \text { if $j>n$}.
\end{cases}
\end{eqnarray*}
 

One of the main results of this section is the following; see \cite[Theorem 3.1]{Zhu}.

\begin{teo}\label{quasiprinc}
Let  $\mathcal U$ and $\mathcal T$ be $\mathrm{Hom}$-finite Krull-Schmidt quasi-hereditary categories with respect  to filtrations $\{\mathcal U_j\}_{0\le j\le n}$,  $\{\mathcal T_i\}_{j\ge 0}$ of $\mathcal U$ and $\mathcal T$, respectively,  consisting of additively  closed subcategories.  Assume  that $M_T\in\mathcal F(_{\mathcal U}\Delta)$ for all $T\in\mathcal T$. Then $\Lambda= \begin{pmatrix} \mathcal T & 0 \\ M & \mathcal{U} \end{pmatrix}$ is quasi-hereditary with respect to the filtration  $\{\Lambda_j\}_{j\ge 0 }$ given in (\ref{Filt1}).
\end{teo}
\bigskip

The proof of  Theorem \ref{quasiprinc} will be a consequence of a series of results that are presented below.

\begin{lem}\label{lem}
	Let $\mathcal U$ be  a quasi-hereditary category with respect to a filtration $\{\mathcal U_j\}_{j\ge 0}$. Let $M$ be a $\mathcal U$-module, and set $M^{[j]}:= \mathrm{Tr}_{\{\mathcal U(U,-)\}_{U\in\mathcal U_j } }M$. In addition, assume that $M\in\mathcal{F}(_\mathcal U\Delta)$. Then for all $U'\in \mathcal{U}$, $M^{[j]}(U')$ $=\{ m: m=M(s)(m')  \text{ for some }  s\in\mathcal U (U'', U'),  \text {with } U'' \in \mathcal{U}_j, \text{ and } m'\in M(U'') \}.$
\end{lem}

\begin{proof}
	 ($\subseteq$). By Yoneda's isomorphism  $Y^{U'}:\mathrm {Nat}\left((U',-), M^{[j]}\right)\cong  M^{[j]}(U')$, 
	 $\eta\mapsto \eta_{U'}(1_{U'})$.   Let  $m\in M^{[j]}(U')$,  then there  exists  $\eta^m:(U',-)\rightarrow M^{[j]}$ such that  $\eta^m_{U'}(1_{U'})=m$. On the other hand, by Lemma \ref{lem5},  there exists  $p: (U'',-) \rightarrow M^{[j]} \rightarrow 0$, with $U''\in \mathcal{U}_j$ and therefore there exists a morphism  $s:U''\rightarrow U'$ for which the following diagram is commutative:
	 
	 \[
	\xymatrix{  & (U',-) \ar@{-->}[ld]_{\exists\hspace{.08cm} (s,-)} \ar[d]^{\eta^m}&  \\ (U'',-) \ar[r]^{p} & M^{[j]}  \ar[r] & 0\ .}
	\]
		Again by Yoneda's lemma we have the following commutative diagram:

	 \begin{equation}
	 \begin{diagram}\label{LemmaQH1}
	 \node{((U'',-),M^{[j]}) }\arrow{e,t}{(s,-)^{*}}\arrow{s,l}{Y^{U''}}   \node{((U',-),M^{[j]})} 	 \arrow{s,l}{Y^{U'}} \\ 
           \node{M^{[j]}(U'')} \arrow{e,t}{M^{[j]}(s)}\node{M^{[j]}(U') \ .}
           \end{diagram}
	 \end{equation}
	 
	Let  $m':=Y^{U''}(p)$. Since $M^{[j]}$ is a subfunctor of $M_T$, we have $m=M^{[j]}(s)(m')=M(s)(m')$,  and we get what we desire. 
		
	( $\supseteq$). Let $m\in  M(U') $  and assume  there exists $s:U''\rightarrow U' $, with $U''\in \mathcal U_j$,  such that  $ m=M(s)(m')$ for some $ m'\in M(U'')$.
	 There then exist  morphisms $\eta^m:(U',-)\rightarrow M$ and  $p^{m'}:(U'',-) \rightarrow M$ such that $\eta^m_{U'}(1_{U'})=m$ and  $p^{m'}_{U''}(1_{U''})=m'$. Thus,  by using the diagram (\ref{LemmaQH1}),  we have  $p^{m'}\circ (g,-)=\eta^m$.  Note that $\mathrm{Im} \hspace{.08cm} p^{m'}$ is a subfunctor of  $M$ and is  generated by  $(U'',-)$. Since  $U''\in \mathcal{U}_j$,  $\mathrm{Im} \hspace{.08cm} p^{m'}$  is contained in the  largest submodule of  $M$  generated by $\{(U,):U\in \mathcal{U}_j\}$, namely $M^{[j]}$:  thus  $\mathrm{Im} \hspace{.08cm} p^{m'} \subset M^{[j]}$.  It follows that 
	  \[
 m=\eta^m_{U'}(1_{U'})=p^{m'}_{U'}\circ (s,-)_{U'}(1_{u'})=p^{m'}_{U'}(s) \in \mathrm{Im} \hspace{.08cm} p^{m'}(U') \subset M^{[j]}(U'):
 \]
that is,   $m \in M^{[j]}(U')$.
\end{proof}


In the remainder of this section, we will assume that the categories $\mathcal U$ and $\mathcal T$ are $\mathrm{Hom}$-finite Krull-Schmidt quasi-hereditary categories with respect to filtrations of additively closed subcategories $\{\mathcal U_j\}_{0\le j\le n}$ and $\{\mathcal T_j\}_{j\ge 0}$ of $\mathcal U$ and $\mathcal T$, respectively, and $M_T\in\mathcal F(_{\mathcal U}\Delta)$  for all $T\in\mathcal T$.

\begin{pro}\label{pro}
	Let  $E=\left(\begin{smallmatrix} T & 0 \\ M & U \end{smallmatrix}\right)$ and $E'=\left(\begin{smallmatrix} T' & 0 \\ M & U' \end{smallmatrix}\right)$ in $\Lambda$. Then, 
	\begin{eqnarray*} I_{\Lambda_j}(E,E')&=&\begin{pmatrix} 0 & 0 \\ M^{[j]}_T(U') & I_{\mathcal{U}_j}(U,U') \end{pmatrix}, \text{if } 0\leq j\leq n, \text{and }\\
	I_{\Lambda_j}(E,E')&=& \begin{pmatrix} I_{\mathcal{T}_{j-n}}(T,T') & 0 \\ M_T(U') & \mathcal{U}(U,U') \end{pmatrix}, \text{ if } j> n.
	\end{eqnarray*}
	\end{pro}

\begin{proof}
	Let $\left(\begin{smallmatrix}f & 0 \\ m & h	\end{smallmatrix}\right)\in  \mathrm{Hom}_{\Lambda}(E,E')$. Therefore,  $f\in \mathrm{ Hom}_{\mathcal{T}}(T,T')$, $m\in M(T,U')$ and $h\in \mathrm{Hom}_{\mathcal{U}}(U,U')$.
	
	 ($0\leq j\leq n$).  The morphism $\left(\begin{smallmatrix}f & 0 \\ m & h \end{smallmatrix}\right)$ lies in $I_{\Lambda_j}(E,E')$ if and only of there is a  commutative diagram 
	\[
    \begin{tikzcd}[ampersand replacement=\&]
   \begin{pmatrix} T & 0 \\ M & U \end{pmatrix} \ar[r, "{\left(\begin{smallmatrix}f & 0 \\ m & h \end{smallmatrix}\right)}"] \ar[dr, "{\left(\begin{smallmatrix}0 & 0 \\ m' & r \end{smallmatrix}\right)}"']\& \begin{pmatrix} T' & 0 \\ M & U' \end{pmatrix}\\ \& \begin{pmatrix} 0 & 0 \\ M & U'' \end{pmatrix} \ar[u, "{\left(\begin{smallmatrix}0 & 0 \\ 0 & s \end{smallmatrix}\right)}"']
    \end{tikzcd}
	\] 
 with  $\left(\begin{smallmatrix} 0 & 0 \\ M & U'' \end{smallmatrix}\right) \in \Lambda_j$, $1\leq j \leq n$. Thus,  $U'' \in \mathcal{U}_j$ and   $\left(\begin{smallmatrix}0 & 0 \\ 0 & s \end{smallmatrix}\right)\left(\begin{smallmatrix}0 & 0 \\ m' & r \end{smallmatrix}\right)$=$\left(\begin{smallmatrix}f & 0 \\ m & h \end{smallmatrix}\right)$;  therefore,  $f=0$, $m=sm'$ and $h=sr$.  It is clear that 
 $h \in I_{\mathcal{U}_j}(U,U')$ because $U'' \in \mathcal{U}_j$. In this way we conclude  that  $\left(\begin{smallmatrix}f & 0 \\ m & h \end{smallmatrix}\right) \in I_{\Lambda_j}(E,E')$  if and only if  $h\in I_{\mathcal{U}_j}(U,U')$ and  $m=sm'=M_T(s)(m')$ where $h: U\xrightarrow{r}U''\xrightarrow {s} U'$ and  $U'' \in \mathcal{U}_j$, in other words, $h\in I_{\mathcal{U}_j}(U,U')$ and  $m\in M_T^{[j]}(U')$, by Lemma \ref{lem}. Thus, $I_{\Lambda_j}(E,E')=\left(\begin{smallmatrix} 0 & 0 \\ M^{[j]}_T(U') & I_{\mathcal{U}_j}(U,U') \end{smallmatrix}\right)$.
 
 ($j>n$) .  $\left(\begin{smallmatrix}f & 0 \\ m & h \end{smallmatrix}\right)\in I_{\Lambda_j}(E,E')$  if and only if there is a commutative diagram
 \[
 \begin{tikzcd}[ampersand replacement=\&]
 \begin{pmatrix} T & 0 \\ M & U \end{pmatrix} \ar[r, "{\left(\begin{smallmatrix}f & 0 \\ m & h \end{smallmatrix}\right)}"] \ar[dr, "{\left(\begin{smallmatrix}r & 0 \\ m_1 & h_1 \end{smallmatrix}\right)}"']
 \& \begin{pmatrix} T' & 0 \\ M & U' \end{pmatrix} \\ \& \begin{pmatrix} T'' & 0 \\ M & U'' \end{pmatrix} \ar[u, "{\left(\begin{smallmatrix} s & 0 \\ m_2 & h_2 \end{smallmatrix}\right)}"']
 \end{tikzcd}
 \]
 with  $\left(\begin{smallmatrix} T'' & 0 \\ M & U'' \end{smallmatrix}\right) \in \Lambda_j$, for $j=n+(j-n)>n$, that is,   $T''\in\mathcal T_{j-n}$ and $U'' \in \mathcal{U}$.
 
 Since  $\left(\begin{smallmatrix}s & 0 \\ m_2 & h_2 \end{smallmatrix}\right)\left(\begin{smallmatrix}r & 0 \\ m_1 & h_1 \end{smallmatrix}\right)=\left(\begin{smallmatrix}f & 0 \\ m & h \end{smallmatrix}\right)$,  we get that $f=sr$, $m=m_2r+h_2m_1$ and $h=h_2h_1$; moreover,  $m\in M(U',T)$, $h\in\mathcal U(U,U')$,   and  $f\in I_{\mathcal{T}_{j-n}}(T,T')$ since  $r\in \mathcal T(T ,  T'')$ and $T''\in\mathcal T_{j-n} $.  Therefore ,  $\left(\begin{smallmatrix}f & 0 \\ m & h \end{smallmatrix}\right)\in I_{\Lambda_j}(E,E')$ if and only if  $m\in M(U',T)$, $h\in\mathcal U(U,U')$,   and  $f\in I_{\mathcal{T}_{j-n}}(T,T')$. 
 
\end{proof}

\begin{pro}\label{onequasi}
	For each pair  $E,E'\in \Ind \hspace{.1cm} \Lambda_j- \Ind \hspace{.1cm} \Lambda_{j-1}$, we have
	\[
	\mathrm{rad}_{\Lambda}(E,E')= I_{ \Lambda_{j-1}}(E,E').
	\]
\end{pro}

\begin{proof}  

The proof is divided in two cases. 


	($1\leq j \leq n$). Let $E=\left (\begin{smallmatrix} 0 & 0 \\ M & U \end{smallmatrix}\right)$ and $E'=\left(\begin{smallmatrix} 0 & 0 \\ M & U' \end{smallmatrix}\right)$ for which  $U,U'\in \mathrm{Ind} \hspace{.1cm} \mathcal{U}_j- \mathrm{Ind} \hspace{.1cm} \mathcal{U}_{j-1}$. Therefore, by \cite[Proposition 3.7]{LGOS1}
	\[
	\footnotesize \mathrm{rad}_{\Lambda}\left(\begin{pmatrix} 0 & 0 \\ M & U \end{pmatrix},\begin{pmatrix} 0 & 0 \\ M & U' \end{pmatrix}\right)=\begin{pmatrix} \mathrm{rad}_{\mathcal{T}}(0,0) & 0 \\ M(U',0) & \mathrm{rad}_{\mathcal{U}}(U,U') \end{pmatrix}=\begin{pmatrix} 0 & 0 \\ 0 & \mathrm{rad}_{\mathcal{U}}(U,U') \end{pmatrix}.
	\]
	On the other hand,  by Theorem \ref{Mar1} we have $\mathrm{rad}_{\mathcal{U}}(U,U')=I_{\mathcal{U}_{j-1}}(U,U')$. Therefore, by  Proposition \ref{pro} we conclude that
	\[
	\footnotesize \mathrm{rad}_{\Lambda}\left(\begin{pmatrix} 0 & 0 \\ M & U \end{pmatrix},\begin{pmatrix} 0 & 0 \\ M & U' \end{pmatrix}\right)=I_{\Lambda_{j-1}}\left(\begin{pmatrix} 0 & 0 \\ M & U \end{pmatrix},\begin{pmatrix} 0 & 0 \\ M & U' \end{pmatrix}\right).
	\]
	
	($j>n$). Let  $E=\left(\begin{smallmatrix} T & 0 \\ M & 0 \end{smallmatrix}\right)$ and $E'=\left(\begin{smallmatrix} T' & 0 \\ M & 0 \end{smallmatrix}\right)$  such that $T,T'\in \mathrm{Ind} \hspace{.1cm} \mathcal{U}_{j-n}- \mathrm{Ind} \hspace{.1cm} \mathcal{U}_{j-n-1}$.  By \cite[Proposition 3.7]{LGOS1} we have
	\[
	\footnotesize \mathrm{rad}_{\Lambda}\left(\begin{pmatrix} T & 0 \\ M & 0 \end{pmatrix},\begin{pmatrix} T' & 0 \\ M & 0 \end{pmatrix}\right)=\begin{pmatrix} \mathrm{rad}_{\mathcal{T}}(T,T') & 0 \\ M(0,T) & \mathrm{rad}_{\mathcal{U}}(0,0) \end{pmatrix}=\begin{pmatrix} \mathrm{rad}_{\mathcal{T}}(T,T') & 0 \\ 0 & 0 \end{pmatrix}.
	\]
	 Again by Theorem \ref{Mar1} we know that  $\mathrm{rad}_{\mathcal{T}}(T,T')=I_{\mathcal{T}_{j-n-1}}(T,T')$.  It follows from Proposition \ref{pro} that
	\[
	\footnotesize \mathrm{rad}_{\Lambda}\left(\begin{pmatrix} T & 0 \\ M & 0 \end{pmatrix},\begin{pmatrix} T' & 0 \\ M & 0 \end{pmatrix}\right)=I_{\Lambda_{j-1}}\left(\begin{pmatrix} T & 0 \\ M & 0 \end{pmatrix},\begin{pmatrix} T' & 0 \\ M & 0 \end{pmatrix}\right).
	\]
\end{proof}

\begin{rk}\label{equivalencecomma}
 Recall, \cite[Theorem 3.14]{LGOS1}, that there exists an equivalence of categories
  $\mathfrak{f}: \left(\mathrm{Mod}\ \mathcal{T},\mathbb{G}(\mathrm{Mod}\ \mathcal{U})\right) \rightarrow  \mathrm{Mod}\ \Lambda$, and, given a $\Lambda$-module $C$, there exists a pair of functors $C_1: \mathcal{T} \rightarrow \mathbf{Ab}$, $C_2: \mathcal{U} \rightarrow \mathbf{Ab}$ and $\xymatrix{C_1 \ar[r]^f & \mathbb{G}(C_2)}$ such that  $\mathfrak{f}((C_1,f,C_2))\cong C$, where $C_1(T'):=C\left(( \begin{smallmatrix} T' & 0 \\ M & 0 \end{smallmatrix} )\right)$ and $C_2(U'):=C\left( (\begin{smallmatrix} 0 & 0 \\ M & U' \end{smallmatrix}) \right)$. Moreover, if there exists exact sequences 
	\[
	\xymatrix{(T_1,-)\ar[r] & (T_0,-)\ar[r] & C_1 \ar[r] & 0}, \hspace{.2cm} T_0,T_1 \in \mathcal{T} \hspace{.3cm} \text{and }
	\]
	\[
    \xymatrix{(U_1,-)\ar[r] & (U_0,-)\ar[r] & C_2 \ar[r] & 0}, \hspace{.2cm} U_0,U_1 \in \mathcal{U}, \hspace{.33cm} 
    \]
then there exists an exact sequence
	\[
\begin{tikzcd}
\left( \left(\begin{smallmatrix} T_1 & 0 \\ M & U_1 \end{smallmatrix}\right) , - \right) \ar[r]
& \left( \left(\begin{smallmatrix} T_0 & 0 \\ M & U_0 \end{smallmatrix}\right) , - \right) \ar[r]  & \mathfrak{f}\left(\left(I_{\Lambda_j}^{(1)},f,I_{\Lambda_j}^{(2)}\right)\right) \ar[r] & 0, \end{tikzcd}
    \]   
   where $P_j:=\left(\left( \begin{smallmatrix} T_j & 0 \\ M & U_j \end{smallmatrix}\right) , - \right)$ is  a projective $\Lambda-$module  for  $j=0,1;$ see proof of   \cite[Proposition 6.3]{LGOS1}.
 \end{rk}

As  a direct result of  the above proposition, we have:

\begin{lem}\label{ILambda1}
Let $X= \left(\begin{smallmatrix} T & 0 \\ M & U \end{smallmatrix}\right)\in\Lambda$. Let us  identify  $I_{\Lambda_j}(X,-)$ with

$\left (I_{\Lambda_j}^{(1)}(X,-), f, I_{\Lambda_j}^{(2)}(X,-)\right )$ in $(\mathrm{Mod}\ \mathcal T, \mathbb{G}(\mathrm{Mod}\ \mathcal U))$, then  

\begin{itemize}
\item [(i)] $ I_{\Lambda_j}^{(1)}(X,-)=0$ and $I_{\Lambda_j}^{(2)}(X,-)\cong M_T^{[j]} \coprod I_{\mathcal{U}_j}(U,-)$, \ if \ \ $0\le j\le n$;
\item [(ii)] $ I_{\Lambda_j}^{(1)}(X,-)\cong I_{\mathcal T_{j-n}}(T,-)$ and $I_{\Lambda_j}^{(2)}(X,-)\cong M_T \coprod \mathcal U(U,-)$, \ if \ \ $j>n$.
\end{itemize}
\end{lem}

\begin{proof}
 Let $T'\in \mathcal{T}$ and $U'\in \mathcal{U}$. The proof is divided into two cases. 

($1\leq j \leq n$). By Proposition  \ref{pro}, we have 
	\[
	I_{\Lambda_j}^{(1)}(T')=I_{\Lambda_j}\left( X, \left( \begin{smallmatrix} T' & 0 \\ M & 0 \end{smallmatrix}\right) \right)=\left(\begin{smallmatrix} 0 & 0 \\ 0 & 0 \end{smallmatrix}\right) \hspace{.1cm} 
	\]
	and
	 \begin{eqnarray*} 
	 I_{\Lambda_j}^{(2)}(U')&=&I_{\Lambda_j}\left( X, \left(\begin{smallmatrix} 0 & 0 \\ M & U' \end{smallmatrix}\right) \right)\\
	                                    &=&\left(\begin{smallmatrix} 0 & 0 \\ M^{[j]}(U') & I_{\mathcal{U}_j}(U,U') \end{smallmatrix}\right) \cong M^{[j]}(U') \coprod I_{\mathcal{U}_j}(U,U') .
	                                    	\end{eqnarray*}

($j>n$).  By  Proposition \ref{pro}, we have
   \[
  I_{\Lambda_j}^{(1)}(T')=I_{\Lambda_j}\left( X, \left(\begin{smallmatrix} T' & 0 \\ M & 0 \end{smallmatrix}\right) \right) = \left(\begin{smallmatrix} I_{\mathcal{T}_{j-n}}(T,T') & 0 \\ M^{[j]}_T(0) & 0 \end{smallmatrix}\right) \cong I_{\mathcal{T}_{j-n}}(T,T') \hspace{.2cm} \small \text{and}
   \]
   \[
  I_{\Lambda_j}^{(2)}(U')=I_{\Lambda_j}\left( X, \left(\begin{smallmatrix} 0 & 0 \\ M & U' \end{smallmatrix}\right) \right)= \left(\begin{smallmatrix} I_{\mathcal{T}_{j-n}}(T,0) & 0 \\ M_T(U') & \mathcal U(U,U') \end{smallmatrix}\right) \cong M_T(U') \coprod \mathcal U(U,U').
   \]	                                    
\end{proof}

\begin{pro}\label{twoquasi}
	Let $X= \left(\begin{smallmatrix} T & 0 \\ M & U \end{smallmatrix}\right)\in \Lambda$, and assume $M_T\in \mathcal{F}({}_{\mathcal{U}}\Delta)$ for all $T\in \mathcal{T}$. Then, for all $j\geq 1$ the following exact sequence exists:
	\[
	\xymatrix{(E_{j-1},-) \ar[r] & (E_{j},-) \ar[r] & I_{\Lambda_j}(X,-)\ar[r]&0},
	\]
	with $E_j\in \Lambda_j$ and $E_{j-1}\in \Lambda_{j-1}$.
\end{pro}

\begin{proof}
	The proof is divided into two cases. 
	
	($1\leq j \leq n$). By Lemma  \ref{lem5}  and Theorem  \ref{Mar1}, there exist  exact sequences 
	\begin{eqnarray*}
		(U'_{j-1},-)\rightarrow  (U'_{j},-)\rightarrow  I_{\mathcal{U}_j}(U,-) \rightarrow 0,\\
		(U''_{j-1},-)\rightarrow (U''_{j},-)\rightarrow   M_T^{[j]}(U,-) \rightarrow   0,
         \end{eqnarray*}
         with  $U'_{j}, U''_{j}\in \mathcal{U}_j$, $U'_{j-1}, U''_{j-1}\in \mathcal{U}_{j-1}$.  Thus, we have an exact sequence  
	$$(U_{j-1},-)\rightarrow (U_{j},-)\rightarrow M_T^{[j]} \coprod I_{\mathcal{U}_j}(U,-)\cong I_{\Lambda_j}^{(2)}(X,-)\rightarrow 0,$$ with $U_{j-1}=U'_{j-1}\coprod U''_{j-1}\in\mathcal U_{j-1}$ and 
	$U_{j}=U'_{j}\coprod U''_{j}\in\mathcal U_{j}$.  It follows that there exists a exact sequence of $\Lambda$-modules
   \[
\left( \left(\begin{smallmatrix} 0 & 0 \\ M & U_{j-1} \end{smallmatrix}\right) , - \right) \rightarrow  \left( \left(\begin{smallmatrix} 0 & 0 \\ M & U_{j} \end{smallmatrix}\right), - \right) \rightarrow  \mathfrak{f}\left((I_{\Lambda_j}^{(1)},f,I_{\Lambda_j}^{(2)})\right) \rightarrow  0,
  \]
  with  $\mathfrak{f}((I_{\Lambda_j}^{(1)},f,I_{\Lambda_j}^{(2)})) \cong I_{\Lambda_j}(X,-)$, $\left(\begin{smallmatrix} 0 & 0 \\ M & U_{j-1} \end{smallmatrix}\right) \in \Lambda_{j-1}$ and  $\left(\begin{smallmatrix} 0 & 0 \\ M & U_{j} \end{smallmatrix}\right) \in \Lambda_j$. 

  \bigskip 

   ($j>n$).  Since $M_T\in\mathcal{F}(\Delta)$, $M_T$ is a finitely presented $\mathcal U$-module, then $M_T\coprod \mathcal U(U,-)\cong  I_{\Lambda_j}^{(2)}(X,-)$ is  finitely presented: an  exact sequence of $\mathcal U$-modules $(U'',-) \rightarrow  (U',-) \rightarrow  I_{\Lambda_j}^{(2)}(X,-)\rightarrow 0 $ exists. On the other hand, an exact sequence of $\mathcal T$-modules exists:   \[
  (T_{j-n-1},-)\rightarrow  (T_{j-n},-)\rightarrow I_{\mathcal{T}_{j-n}}(T,-) \rightarrow 0, 
   \]
   with $T_{j-n}\in \mathcal{T}_{j-n}$ and $T_{j-n-1}\in \mathcal{T}_{j-n-1}$. Thus, we get an exact sequence
   \[
   \begin{tikzcd}
   \left( \left(\begin{smallmatrix} T_{j-n-1} & 0 \\ M & U'' \end{smallmatrix}\right) , - \right) \rightarrow  \left( \left(\begin{smallmatrix} T_{j-n} & 0 \\ M & U' \end{smallmatrix}\right), - \right) \rightarrow  \mathfrak{f}\left((I_{\Lambda_j}^{(1)},f,I_{\Lambda_j}^{(2)})\right) \rightarrow  0, 
   \end{tikzcd}
   \]
   $\mathfrak{f}((I_{\Lambda_j}^{(1)},f,I_{\Lambda_j}^{(2)}))\cong I_{\Lambda_j}(X,-)$, $\left(\begin{smallmatrix} T_{j-n-1} & 0 \\ M & U'' \end{smallmatrix}\right) \in \Lambda_{j-1}$ and  $\left(\begin{smallmatrix} T_{j-n} & 0 \\ M & U' \end{smallmatrix}\right) \in \Lambda_j$.
\end{proof}

\bigskip

\begin{proof}[Proof of of Theorem \ref{quasiprinc} ]
It follows from Propositions \ref{onequasi} and  \ref{twoquasi}.
\end{proof}

\subsection{ The standard modules in $\mathrm{Mod}\ \Lambda$ and $\mathcal{F}(_\Lambda \Delta)$.}

Assume that $\mathcal{T}$ and $\mathcal{U}$ are $\mathrm{Hom}$-finite Krull-Schmidt and
quasi-hereditary $K-$categories with respect to filtrations $\{\mathcal U_j\}_{0\le j\le n}$ and  $\{\mathcal T_j\}_{j\ge 0}$, respectively. If in addition $M: \mathcal U\otimes_K\mathcal T^{op}\rightarrow \mathrm{mod}\ K$ is a functor such that $M_T=M(-,T):\mathrm{Mod}\ \mathcal U\rightarrow \mathrm{mod}\ K$ is finitely presented $\mathcal U$-module, then by Theorem \ref{quasiprinc}, the triangular matrix category $\Lambda=(\begin{smallmatrix} \mathcal T&0\\ M&\mathcal U\end{smallmatrix})$ is a quasi-hereditary $K$-category with respect to some filtration $\{\Lambda_j\}_{j\ge 0}$. In this part, we study the relation between the full standard subcategories $_\mathcal U\Delta$, $_\mathcal T\Delta$, and $_\Lambda\Delta$ of $\mathrm{Mod}\ \mathcal U$,   $\mathrm{Mod}\ \mathcal T$ and  $\mathrm{Mod}\ \Lambda$, respectively. More concretely, we will show in Theorem \ref{filt1}, by using the notation of Remark \ref{equivalencecomma},  that 
   $$\mathcal{ F}_f(_{\Lambda}\Delta)=\{(F^{(1)}, f, F^{(2)}): F^{(1)}\in\mathcal F_f(_{\mathcal T}\Delta) \text { and } F^{(2)}\in\mathcal F_f(_{\mathcal U}\Delta)\}; $$
see \cite[Theorem 3.1]{Zhu}.

\bigskip

Regardless, we need the following result.
 
 \begin{pro}\label{standard1}
	The functor $\mathfrak{f}:(\mathrm{Mod} \ \mathcal{T}, \mathbb{G}(\mathrm{Mod}\ \mathcal{U})) \rightarrow \mathrm{Mod}\ \Lambda$ induce equivalences of full subcategories:
	\[
	(0,0,_{\mathcal{U}}\Delta(j)) \longleftrightarrow _{\Lambda}\Delta(j), \hspace{.3cm} \text{if} \hspace{.2cm} 1\leq j \leq n, \hspace{.3cm} \text{and},
	\]
	\[
	(_{\mathcal{T}}\Delta(j-n),0,0) \longleftrightarrow _{\Lambda}\Delta(j), \hspace{.3cm} \text{if} \hspace{.2cm} j>n. \hspace{1.1cm}
	\]
\end{pro}

\begin{proof}
  First, we identify $_{\Lambda}\Delta_E=\frac{\Lambda(E,-)}{I_{\Lambda_{j-1}}(E,-)}$, $E \in \mathrm{Ind} \hspace{.1cm} \Lambda_j - \mathrm{Ind} \hspace{.1cm} \Lambda _{j-1}$,   with a triple $( \Delta_E^{(1)}, g, \Delta_E^{(2)})$ with 
 $\Delta_E^{(1)}:\mathcal{T}\rightarrow \mathbf{Ab}$ and $\Delta_E^{(2)}:\mathcal{U}\rightarrow \mathbf{Ab}$.
 
 \bigskip
 
 ($1\leq j \leq n$). Let $T'\in \mathcal{T}$ and $E= \left(\begin{smallmatrix} 0 & 0 \\ M & U \end{smallmatrix}\right)$ with $U \in \mathrm{Ind} \hspace{.1cm} \mathcal{U}_j - \mathrm{Ind} \hspace{.1cm} \mathcal{U}_{j-1}$. Then
\[
_{\Lambda}\Delta_E^{(1)}(T')=\frac{\Lambda \left( \left(\begin{smallmatrix} 0 & 0 \\ M & U \end{smallmatrix}\right), \left(\begin{smallmatrix} T' & 0 \\ M & 0 \end{smallmatrix}\right) \right)}{I_{\Lambda_{j-1}}\left( \left(\begin{smallmatrix} 0 & 0 \\ M & U \end{smallmatrix}\right), \left(\begin{smallmatrix} T' & 0 \\ M & 0 \end{smallmatrix}\right) \right)}\cong 0.
\]
On the other hand, if  $U' \in \mathcal{U}$, we get 
\[
_{\Lambda}\Delta_E^{(2)}(U')=\frac{\Lambda \left( \left(\begin{smallmatrix} 0 & 0 \\ M & U \end{smallmatrix}\right), \left(\begin{smallmatrix} 0 & 0 \\ M & U \end{smallmatrix}\right) \right)}{I_{\Lambda_{j-1}}\left( \left(\begin{smallmatrix} 0 & 0 \\ M & U \end{smallmatrix}\right), \left(\begin{smallmatrix} 0 & 0 \\ M & U \end{smallmatrix}\right) \right)}\cong \frac{\mathcal{U}(U,U')}{I_{\mathcal{U}_{j-1}}(U,U')}.
\]
In this way,  $_{\Lambda}\Delta_E^{(1)}\cong 0$ and $_{\Lambda}\Delta_E^{(2)}\cong \frac{\mathcal{U}(U,-)}{I_{\mathcal{U}_{j-1}}(U,-)}=_{\mathcal U}\Delta_U$, with $U \in \mathrm{Ind} \hspace{.1cm} \mathcal{U}_j - \mathrm{Ind} \hspace{.1cm} \mathcal{U}_{j-1}$.

($j> n$).  Let $T' \in \mathcal{T}$ and $E= \left(\begin{smallmatrix} T & 0 \\ M & 0 \end{smallmatrix}\right)$  with  $T \in \mathrm{Ind} \hspace{.1cm} \mathcal{T}_{j-n} - \mathrm{Ind} \hspace{.1cm} \mathcal{T}_{j-n-1}$. Thus
\[
_{\Lambda}\Delta_E^{(1)}(T')=\frac{\Lambda \left( \left(\begin{smallmatrix} T & 0 \\ M & 0 \end{smallmatrix}\right), \left(\begin{smallmatrix} T' & 0 \\ M & 0 \end{smallmatrix}\right) \right)}{I_{\Lambda_{j-1}}\left( \left(\begin{smallmatrix} T & 0 \\ M & 0 \end{smallmatrix}\right), \left(\begin{smallmatrix} T' & 0 \\ M & 0 \end{smallmatrix}\right) \right)}\cong \frac{\mathcal{T}(T,T')}{I_{\mathcal{T}_{j-n-1}}(T,T')}.
\]
If $U' \in \mathcal{U}$, we obtain
\[
_{\Lambda}\Delta_E^{(2)}(U')=\frac{\Lambda \left( \left(\begin{smallmatrix} T & 0 \\ M & 0 \end{smallmatrix}\right), \left(\begin{smallmatrix} 0 & 0 \\ M & U' \end{smallmatrix}\right) \right)}{I_{\Lambda_{j-1}}\left( \left(\begin{smallmatrix} T & 0 \\ M & 0 \end{smallmatrix}\right), \left(\begin{smallmatrix} 0 & 0 \\ M & U' \end{smallmatrix}\right) \right)}\cong 0.
\]
Therefore $_{\Lambda}\Delta_E^{(1)}\cong \frac{\mathcal{T}(T,-)}{I_{\mathcal{T}_{j-n-1}}(T,-)} \cong _{\mathcal T}\Delta _T$, with $T \in \mathrm{Ind} \hspace{.1cm} \mathcal{T}_{n-j} - \mathrm{Ind} \hspace{.1cm} \mathcal{T}_{n-j-1}$  and $_{\Lambda}\Delta_E^{(2)}\cong 0$.

\end{proof}

\begin{teo}\label{filt1}
Let $F=(F^{(1)}, f, F^{(2)})\in\mathrm{Mod}\ \Lambda$, and consider its trace filtration  $\{F^{[j]} \}_{j\ge 0}= \{(F^{[j]})^{(1)}, f^{[j]}, (F^{[j]})^{(2)})\} _{j\ge 0}$  with respect to $\{\Lambda_j\}$. Then: 
\begin{itemize}
\item [(i)] $(F^{[j]})^{(1)}\cong 0$,  if $0\le j\le n$, and $(F^{[j]})^{(2)}\cong F^{(2)}$, if $j>1$.
\item [(ii)] If $F=(F^{(1)},f,F^{(2)})\in\mathcal{ F}(_{\Lambda}\Delta)$ then  $F^{(1)}\in\mathcal F(_{\mathcal T}\Delta)$ and $F^{(2)}\in\mathcal F(_{\mathcal U}\Delta)$.
\item[(iii)] $\mathcal{ F}_f(_{\Lambda}\Delta)=\{(F^{(1)}, f, F^{(2)}): F^{(1)}\in\mathcal F_f(_{\mathcal T}\Delta) \text { and } F^{(2)}\in\mathcal F_f(_{\mathcal U}\Delta)\}$.
\end{itemize}
\end{teo}

\begin{proof}
(i) Assume we have an exact sequence of $\Lambda$-modules 
(*) $(X',-)\rightarrow (X,-) \rightarrow  F\rightarrow 0$ with  $X'=\left(\begin{smallmatrix} T' & 0 \\ M & U' \end{smallmatrix}\right)$ and  $X=\left(\begin{smallmatrix} T & 0 \\ M & U \end{smallmatrix}\right)$. Thus, we get an exact sequence
(** ) $I_{\Lambda_j}(X',-)\rightarrow I_{\Lambda_j}(X,-)\rightarrow  F^{[j]}\rightarrow 0$. First, we  identify  (*) with  the exact sequence
$$\xymatrix{ & \big (\mathcal{T}(T',-), g',  M_{T'}\amalg \mathcal U(U',-)\big) \ar `[ld] `[] `[ld] |{} `[] [ldr] & \\
&\big (\mathcal{T}(T,-), g,  M_{T}\amalg \mathcal U(U,-))\ar[r] & \big((F^{(1)},f,F^{(2)}\big )\ar[r] &0.   }$$

In particular we have the exact sequence
 \begin{equation}\label{sePROJ}
  M_{T'}\amalg \mathcal{U}(U',-)\rightarrow M_{T}\amalg \mathcal{U}(U,-)\rightarrow F^{(2)}\rightarrow 0.
 \end{equation}
 
Secondly, we identify the exact sequence (**) with the exact sequence  $$\xymatrix{ & \big (I_{{ \Lambda}_{j}}^{(1)}(X',-), h',  I_{{ \Lambda}_{j}}^{(2)}(X',-)\big) \ar `[ld] `[] `[ld] |{} `[] [ldr] & \\
&\big (I_{{ \Lambda}_j}^{(1)}(X,-), h,  I_{{ \Lambda}_j}^{(2)}(X,-)\big)\ar[r] & \big((F^{[j]})^{(1)},f^{[j]},(F^{[j]})^{(2)}\big )\ar[r] &0.   }$$ We then  have exact sequences 
$$ I_{{ \Lambda}_j}^{(k)}(X',-)\rightarrow I_{{ \Lambda}_j}^{(k)}(X,-)\rightarrow (F^{[j]})^{(k)}\rightarrow 0,\ k=1,2 \ .$$

By Lemma \ref{ILambda1} we have that  $ I_{{ \Lambda}_j}^{(1)}(X',-)\cong 0$ and $ I_{{ \Lambda}_j}^{(1)}(X,-) \cong 0$ if $0\le j\le n$; therefore,  $(F^{[j]})^{(1)}\cong 0$. On the other hand, if $j>n$ we have 
$ I_{{ \Lambda}_j}^{(2)}(X',-)\cong M_{T'}\amalg \mathcal U(U',-)$ and $ I_{{ \Lambda}_j}^{(2)}(X,-)\cong M_{T}\amalg \mathcal U(U,-)$. Thus,  by (\ref{sePROJ}), we get $ (F^{[j]})^{(2)}\cong F^{(2)}$ if $j>n$. 

In this way, if $1\le j\le n$ we obtain that  $F^{[j]}/F^{[j-1]}$ is a sum of copies of elements in $_{\Lambda} \Delta(j)$ and
\[
\frac{F^{[j]}}{F^{[j-1]}}\cong
\begin{cases}
\left(0,0, (F^{[j]})^{(2)}/(F^{[j-1]})^{(2)}\right), \text{ if $1\le j\le n$};\\
\left((F^{[j]})^{(1)}/(F^{[j-1]})^{(1)},0, 0\right), \text{ if $j>1$}.
\end{cases}
\]
(ii) follows from by Proposition \ref{standard1}.

(iii) Let $F=(F^{(1)},f,F^{(2)})\in\mathcal{F}_f(_\Lambda \Delta)$. By (ii),  it only remains to prove  that if  $F^{(1)}\in\mathcal{F}_f(_{\mathcal T}\Delta)$ and $F^{(2)} \in\mathcal{F}_f(_{\mathcal U}\Delta)$,  then   $F\in\mathcal{F}_f(_\Lambda \Delta)$. In fact,  the $\Lambda$-modules $(F^{(1)},0,0)$ and $(0,0,F^{(2)})$ are in  $\mathcal{F}_f(_{\Lambda}\Delta)$ by  Proposition \ref{standard1}.   It follows that we have a short exact sequence 
\[
0\rightarrow  (F^{(1)},0,0)\rightarrow (F^{(1)},f,F^{(2)})\rightarrow( 0,0,F^{(2)})\rightarrow 0 \ .
\]
Thus, $ (F^{(1)},f,F^{(2)})$ is  in $\mathcal{F}_f(_{\Lambda}\Delta)$ since $\mathcal{F}_f(_{\Lambda}\Delta)$  is closed under extensions by Remark \ref{remark1}.
\end{proof}

We end this section continuing with Example  \ref{ex1}.

\begin{ex}
Consider the quivers $\small
R:
\begin{tikzcd}[cramped, sep=normal]
	1  \arrow[r, bend left, "\alpha_1", pos=0.46] & 2 \arrow[r, bend left, "\alpha_2"] \arrow[l, bend left, "\beta_1"] & 3 \arrow[r, bend left, pos=0.44, "\alpha_3"] \arrow[l, bend left, "\beta_2"] & \arrow[l, bend left, pos=0.56, "\beta_3"] & \arrow[l,phantom,"\cdots"]  
\end{tikzcd}
$ and $\small
Q:
\begin{tikzcd}[cramped, sep=normal]
1'  & 2' \arrow[r,"\tiny{\gamma_2}"] \arrow[l,"\gamma_1"']& 3' & 4' \arrow[r,"\gamma_4"] \arrow[l,"\gamma_3"'] & 5' \arrow[r,phantom,"\cdots"] &  \hspace{.01cm} .  
\end{tikzcd}
$
Let  $\mathcal U=K\mathcal{Q}$ and $\mathcal T=K\mathcal{R}/\mathcal J$ be the path categories of the above quivers, where $\mathcal J$ is the ideal in $K\mathcal{R}$ generated by the set of relations 
\begin{equation}\label{relation2}
\{\beta_1\alpha_1 \text{ and }  \alpha_{t+1}\alpha_t, \beta_{t}\beta_{t+1}, \alpha_t\beta_t-\beta_{t+1}\alpha_{t+1}, t\ge 1\},
\end{equation}

First, we see that $\mathcal T$ and $\mathcal U$ are quasi-hereditary categories. 

Set $\mathcal T_0=\{0\}$, and let $\mathcal{T}_j=\mathrm{add} \{ t: 1\le t\le j\}$, for $j\ge 1$. Therefore,  $\{0\}=\mathcal T_0\subset \mathcal T_1\subset \cdots$ is a filtration of $\mathcal T$ into additively closed subcategories. (i) It is clear that $\mathrm{rad}_\mathcal T(1,1)=0$ because $\beta_1\alpha_1=0$. Since  $\mathrm{Ind} \mathcal T_j-\mathrm{Ind} \mathcal T_{j-1}=\{j\}$ for all $j\ge 1$, we have that $\mathrm{rad}_\mathcal T(j,j)=(\beta_j\alpha_j)=
(\alpha_{j-1}\beta_{j-1})=I_{\mathcal T_{j-1}}(j,j)$. (ii) $I_{\mathcal T_1}(1,-)\cong(1,-)$, $I_{\mathcal T_1}(2,-)\cong (1,-)$ and $I_{\mathcal T_1}(j,-)=0$, if $j\ge 3$. For $j\ge 2$,  we can readily check that there exists an exact sequence
$ 0\rightarrow I_{\mathcal T_{j-1}}(j,-)\rightarrow \mathcal T(j,-) \rightarrow I_{\mathcal T_{j}}(j+1,-)\rightarrow 0$ and 
$I_{\mathcal T_j}(j+t,-)=0$ if $t\ge 2$.

\[
\begin{tikzcd}[cramped, row sep=3.1em, column sep=2.8em]
	&&&& 0 \arrow[d] \\
	\hspace{.01cm} & \arrow[l,phantom,"\cdots"] K \arrow[r, bend left=22, "1"] \arrow[d,"1"'] & K \arrow[r, bend left=22] \arrow[d,"\tiny{\begin{pmatrix} 0 \\ 1 \end{pmatrix}}"', pos=0.46] \arrow[l, bend left=22, "1"'] & 0\arrow[l, bend left=22] \arrow[d] \arrow[r,phantom,"\cdots" description] & I_{\mathcal T_{j-1}}(j,-) \arrow[d]\\
	\hspace{.01cm} & \arrow[l,phantom,"\cdots"] K \arrow[r, bend left=22, "\tiny{\begin{pmatrix} 0 \\ 1 \end{pmatrix}}", pos=0.48] \arrow[d,"1"'] & K^2 \arrow[r, bend left=22, "\tiny{\begin{pmatrix} 1 & 0 \end{pmatrix}}"] \arrow[d,"\tiny{\begin{pmatrix} 1 & 0 \end{pmatrix}}"', pos=0.66] \arrow[l, bend left=22, "\tiny{\begin{pmatrix} 1 & 0 \end{pmatrix}}"] & K \arrow[l, bend left=22, "\tiny{\begin{pmatrix} 0 \\ 1 \end{pmatrix}}"] \arrow[d,"1"']  \arrow[r,phantom,"\cdots \hspace{.5cm}" description] & \mathcal{T}(j,-) \arrow[d]\\
	\hspace{.01cm} & \arrow[l,phantom,"\cdots"] 0 \arrow[r, bend left=22] & K \arrow[r, bend left=22, "1"'] \arrow[l, bend left=22] & K\arrow[l, bend left=22, "0"]  \arrow[r,phantom,"\cdots" description] & I_{\mathcal T_j}(j+1,-) \arrow[d]\\
	& j-1 & j & j+1  & 0		
\end{tikzcd}
\]

Set $\mathcal U_0=\{0\}$ and  $\mathcal{U}_1=\mathrm{add}\{ j' \in \mathbb{N}: \text{ $j$ is odd }\}$ and $\mathcal{U}_2=\mathrm{add}\{ j': j \in \mathbb{N}\}=K\mathcal Q$. Thus 
$K\mathcal Q$ is quasi-hereditary with respect to the finite  filtration $\{0\}=\mathcal U_0\subset \mathcal U_1\subset \mathcal U_2=K\mathcal Q$. The condition (i) clearly holds since $\mathrm{rad}_\mathcal U(E,E')=I_{\mathcal U_{j-1}}(E,E')=0$ for all pairs $E,E'\in\mathrm{Ind}\ \mathcal {U}_j-\mathrm{Ind}\ \mathcal {U}_{j-1}$ and $j=1,2$. On the other hand, $I_{\mathcal U_1}(j,-)=(j,-)$ if $j$ is odd and $I_{\mathcal U_1}(j,-)\cong (j-1,-)\oplus  (j+1,-)$ if $j$ is even. 

Finally,  the functor $M$ given in Example \ref{ex1}  satisfies $M_T:K\mathcal Q\rightarrow \mathrm{mod}\ K$ is projective for all $T$ since
$M_{1}\cong\mathcal U(1,-)^2$, $M_2\cong \mathcal U(1,-)$, and $M_t\cong 0$,  for all $t> 2$,  which are all  in $\mathcal F(_\mathcal U\Delta)$ because $\mathcal U$ is quasi-hereditary.

In this way,  the matrix category
$\left(
\begin{smallmatrix}
\mathcal T&0\\
M&\mathcal U
\end{smallmatrix}
\right)$ is equivalent to the path category of the quiver $Q''$ modulo the ideal generated by the set of relations  (\ref{relation2}) and is quasi-hereditary with respect to the filtration $\{0\}=\Lambda _0\subset \Lambda _1\subset \Lambda _2\subset \Lambda_3\cdots$, where 
$\Lambda_1=\mathcal U_1$, $\Lambda_2=\mathcal U_2$ and $\Lambda_{j+2}=
\mathrm{add}\left ( \{ j': j \in \mathbb{N}\}\cup \{ t\in\mathbb N: 1\le t\le j\} \right): $
\[
\small \begin{tikzcd}[cramped, sep=normal]
  Q'': \ \ \   \arrow[r,phantom,"\cdots"] &   \arrow[r,"\gamma_3"'] & 3' & 2' \arrow[l,"\tiny{\gamma_2}"] \arrow[r,"\gamma_1"']  & 1' &
	1 \arrow[l,"\varphi"]\arrow[r, bend left, "\alpha_1", pos=0.4] &  \arrow[r, bend left, "\alpha_2"] \arrow[l, bend left, "\beta_1"] & 3 \arrow[r, bend left, "\alpha_3"] \arrow[l, bend left, "\beta_2", pos=0.45] & \arrow[l, bend left, "\beta_3"] & \arrow[l,phantom,"\cdots"]   .
\end{tikzcd}
\]
\end{ex}

\section{One-Point Extensions by Projectives and Classical Tilting}
Let us consider $Q=(Q_0,Q_1)$ that is a  strongly locally  finite quiver and   $*$ is a source in $Q$.  Let $Q'=(Q'_0,Q_1')$ be the quiver obtained by removing the vertex $*$  from $Q$ . Let  $I\subset KQ$,  and set  $\mathcal C=K\mathcal Q/\mathcal I$. 
We then have two full subcategories of $\mathcal C$; 
$\mathcal U=\mathrm{add} \ Q_0'$ and $\mathcal T=\mathrm{add} \{* \}$.   Let
$M:\mathcal U\otimes \mathcal T^{op}\rightarrow \mathrm{mod} \ K$ be the additive functor given by $M(u, *)=\mathcal C(*,u)$,  where 
$(u,*)\in Q_0'\times \{*\}$, and given
 $(g\otimes f)\in  \mathcal  U(U,U')\otimes _K\mathcal T (*,*)
 =(\mathcal U\otimes _K\mathcal T^{op})\left ((U,*),((U',*)\right )$, 
 $M(g\otimes f): M(U,*)\rightarrow M(U',*)$  is defined by
$\mathcal C(f,g):\mathcal C(*,U)\rightarrow \mathcal C(*,U')$ for all $U,U'\in\mathcal U$.

 Consider the triangular matrix category
  $\Lambda=\left(\begin{smallmatrix}
 \mathcal T&0\\
 M&\mathcal U
 \end{smallmatrix}\right )$. Let $X=\left (\begin{smallmatrix}
 T&0\\
 M&U
  \end{smallmatrix}\right) $, $X'=\left(\begin{smallmatrix}
 T'&0\\
 M&U'
  \end{smallmatrix}\right )$ and
  $X''=\left(\begin{smallmatrix}
 T''&0\\
 M&U'' 
  \end{smallmatrix}\right)$ be objects in $\Lambda$. The  set of morphisms 
  $\Lambda (X,X')= \left (\begin{smallmatrix}
 \mathcal T (T,T')&0\\
 M(U',T)&\mathcal U (U,U')
 \end{smallmatrix}\right )$  can be simply written as $\left(\begin{smallmatrix}
 \mathcal  C (T,T')&0\\
 \mathcal C (T,U')&\mathcal C (U,U')
 \end{smallmatrix}\right )$. Thus, the  composition  $\Lambda(X',X'')\times \Lambda(X,X')\rightarrow \Lambda(X,X'')$ in $\Lambda$ is given by the map   
 \begin{eqnarray*}
 \left ( \begin{smallmatrix}
 \mathcal  C (T',T'')&0\\
 \mathcal C (T',U'')&\mathcal C (U',U'')
 \end{smallmatrix}\right ) \times \left (\begin{smallmatrix}
 \mathcal  C (T,T')&0\\
 \mathcal C (T,U')&\mathcal C (U,U')
 \end{smallmatrix}\right )&\longrightarrow& \left ( \begin{smallmatrix}
 \mathcal  C (T,T'')&0\\
 \mathcal C (T,U'')&\mathcal C (U,U'')
 \end{smallmatrix}
  \right )\\
  {}\\
   \left( \left (\begin{smallmatrix}
   f'&0\\
 h'&g'
 \end{smallmatrix}\right ) , \left (\begin{smallmatrix}
   f&0\\
 h&g
 \end{smallmatrix}\right ) \right)&\longmapsto&  \left (\begin{smallmatrix}
   f'f&0\\
 h'f+g'h&g'g
 \end{smallmatrix}\right )
  \end{eqnarray*}
 
  \bigskip
  
 In this way, there is an isomorphism of categories $\mathcal C\cong \Lambda$. 
   \bigskip
   
  A pair of modules  $X$ and $Y$ is called \emph{orthogonal} if $\mathrm{Ext}^ j_\Lambda(X, Y) =\mathrm{Ext}^ j_\Lambda(Y, X)$ $= 0$ for all $j \ge 1$ and \emph{exceptional} if additionally $\mathrm{pd}_\Lambda X <\infty$ for all $X\in\mathrm{mod} \ \Lambda$.
  \bigskip
  
  We then obtain the main result in this section, which is similar to that given in \cite{Assem1}.

\begin{teo}\label{mainExt}
An adjoint  pair of  additive  functors 
$\mathcal R:\mathrm{mod}\ \Lambda\rightarrow \mathrm{mod}\ \mathcal U$ and $\mathcal E:\mathrm{mod}\ \mathcal U\rightarrow \mathrm{mod}\ \Lambda$ exists, which preserve orthogonality and exceptionality .
\end{teo}
  
    Since $Q$ is locally  finite, there is a finite number of neighbours, $u_1, \ldots, u_n$  of $\{*\}$.  Set $U_0:=\oplus_{i=1}^nu_i$,  and consider the morphism $\tilde h:*\rightarrow U_0$ induced by the arrows $\alpha_i:*\rightarrow u_i$, $1\le i\le n$. Thus, for all $u\in Q_0'$ we have an isomorphism
   $\mathcal C(U_0,u) \cong \mathcal C(*, u)$,  $g\mapsto g\tilde h$, which induces an
   isomorphism of $\mathcal U$-modules:
   \begin{equation}\label{isoproj}
   \mathcal U(U_0,-) \cong \mathcal C(*, -)|_{\mathcal U}.
   \end{equation}
      
  Consider the projective $\Lambda$-modules $P=\left(\left(\begin{smallmatrix} * & 0 \\ M & 0 \end{smallmatrix}\right),-\right)$ and $P_0=\left(\left(\begin{smallmatrix} 0 & 0 \\ M & U_0\end{smallmatrix}\right),-\right)$. Thus, given  $X=\left(\begin{smallmatrix}
 T&0\\
 M&U
  \end{smallmatrix}\right)\in\Lambda $, we have an isomorphism of $K$-modules:
   \begin{eqnarray*}
  \varphi_X:P_0(X)&=&\begin{pmatrix} 0 & 0 \\ 0 & \mathcal C(U_0,U) \end{pmatrix}\rightarrow \mathrm{rad} P(X)=\begin{pmatrix} 0 & 0 \\ \mathcal C(*,U) & 0 \end{pmatrix}\\
                         && \begin{pmatrix} 0 & 0 \\ 0 & g\end{pmatrix}\mapsto \begin{pmatrix} 0 & 0 \\ 0 & g \end{pmatrix}\begin{pmatrix} 0 & 0 \\ \tilde{h} & 0 \end{pmatrix}=\begin{pmatrix} 0 & 0 \\ g\tilde{h} & 0\end{pmatrix},
  \end{eqnarray*}
  which is natural in $X$. Thus, we get an isomorphism $\varphi: P_0\rightarrow \mathrm{rad} P$ of $\Lambda$-modules.
  
  Set $S=\frac{P}{\mathrm{rad} P}$. It is clear that  $S$ is a simple  $\Lambda-$module, and it is injective in $(\Lambda, \mathrm{mod}\ K)$ since $\footnotesize S\cong D\left(-, \left(\begin{smallmatrix} * & 0 \\ M & 0 \end{smallmatrix}\right)\right)$, where $D$ is the standard duality. Therefore, we have an exact sequence
   \begin{equation}
   \xymatrix{0 \ar[r] & P_0 \ar[r] & P \ar[r] & S \ar[r] & 0}.
   \end{equation}
  
In what follows, we will refer to objects of $( \U, \mathrm{mod} \ K)$ and $(\Lambda, \mathrm{mod} \ K)$) simply as $\mathcal U$-modules and $\Lambda$-modules, respectively.
  
  We now define a pair of additive   functors, $\footnotesize\mathcal{E}:(\U, \mathrm{mod} \ K) \rightarrow (\Lambda, \mathrm{mod} \ K)$  and   $\footnotesize\mathcal{R}: (\Lambda, \mathrm{mod} \ K)\rightarrow (\U, \mathrm{mod} \ K)$ called \emph{extension} and \emph{restriction}, respectively , and they are defined as follows.
  
 \bigskip
  
  Let $G$ be a $\mathcal U$-module, and consider the pair of objects  $\left(\begin{smallmatrix} 0 & 0 \\ M & U \end{smallmatrix}\right)$ and  $\left(\begin{smallmatrix} * & 0 \\ M & 0 \end{smallmatrix}\right)$ in $\Lambda$. Thus,
\begin{eqnarray*}
(\mathcal{E} G) \begin{pmatrix} * & 0 \\ M & 0 \end{pmatrix} &:=& \mathrm{Hom}_{\U}\left(\restr{\C(*,-)}{\U},G\right) \cong \mathrm{Hom}_{\U}\left(\mathcal U(U_0,-),G\right)\cong G(U_0) \\
 (\mathcal{E} G) \begin{pmatrix} 0 & 0 \\ M & U \end{pmatrix} &:= &\mathrm{Hom}_{\U}\left(\mathcal U(U,-),G\right)\cong G(U).
\end{eqnarray*} 
Let  $F$ be a $\Lambda$-module and  $ U\in \U$. Therefore, $(\mathcal{R}F)(U):=F\left(\begin{smallmatrix} 0 & 0 \\ M & U \end{smallmatrix}\right).$

\bigskip

Clearly, $(\mathcal R, \mathcal E) $  is an adjoint pair of functors, and  both are exact.  

\begin{pro}\label{fp1}
The functor $\mathcal R$ sends finitely presented $\Lambda$-modules into finitely presented $\mathcal U$-modules. 
\end{pro}

\begin{proof}
Set
$Q_*:=\Lambda \left(\left(\begin{smallmatrix} * & 0 \\ M & 0 \end{smallmatrix}\right),-\right)$ and $Q_U:=\Lambda \left(\left(\begin{smallmatrix} 0& 0 \\ M & U \end{smallmatrix}\right),-\right)$. Thus, for all $\overline{U}\in\mathcal U$ we have isomorphisms of $K$-vector spaces $\mathcal U(U_0,\overline U)\cong  \mathcal U(*,\overline U)\cong \mathcal RQ_{*}(\overline U), \theta\mapsto  \left(\begin{smallmatrix} 0 & 0 \\ \theta\tilde{h} & 0 \end{smallmatrix}\right)$ and  $\mathcal U(U,\overline U) \cong \mathcal RQ_{U}(\overline U), \theta\mapsto  \left(\begin{smallmatrix} 0 & 0 \\ 0 & \theta\end{smallmatrix}\right)$, 
which are natural in $\overline{U}$. Thus, we have isomorphisms of $\mathcal U$-modules:
$$\mathcal RQ_U\cong \mathcal U(U,-) \text { and } \mathcal RQ_*\cong \mathcal U(U_0,-). $$Therefore, the restriction $\mathcal R Q$ of a projective finitely generated  $\Lambda$-module  $Q$ is a projective finitely generated $\mathcal U$-module. The rest of the proof follows from the fact that  $\mathcal R$ is an exact functor.
\end{proof}

Later we will see that the extension  functor $\mathcal E$  sends finitely presented \hspace{.6cm} $\mathcal U$-modules into finitely presented $ \Lambda$-modules.
\bigskip 

Given a $\mathcal U$-module $G$, we can see it as a $\Lambda$-module by setting
 $G\left(\begin{smallmatrix} T & 0 \\ M & U \end{smallmatrix}\right)$ $= G(U)$. In this way 
we consider $(\mathcal U, \mathrm{mod} \ K)$  embedded in $(\Lambda, \mathrm{mod} \ K)$ by  identifying  it with the  full subcategory of $(\Lambda, \mathrm{mod} \ K)$  consisting of the $\Lambda-$modules such that $\footnotesize G\left(\begin{smallmatrix} T & 0 \\ M & U \end{smallmatrix}\right)= G \left(\begin{smallmatrix} 0 & 0 \\ M & U \end{smallmatrix}\right)$.
  
Let $X$ be a $\Lambda$-module.  We claim that  $\mathcal{R}X$ is a submodule  of  $X$. In fact,
we have that 
\[
\footnotesize X\begin{pmatrix} * & 0 \\ M & U \end{pmatrix}=X\begin{pmatrix} * & 0 \\ M & 0 \end{pmatrix} \coprod X\begin{pmatrix} 0 & 0 \\ M & U \end{pmatrix}\cong X\begin{pmatrix} * & 0 \\ M & 0 \end{pmatrix} \coprod (\mathcal{R}X)\begin{pmatrix} * & 0 \\ M & U \end{pmatrix},
\]
for all $\left(\begin{smallmatrix} * & 0 \\ M & U \end{smallmatrix}\right)\in \Lambda$. Thus,  $\mathcal{R}$ is a subfunctor of the identity functor $\mathrm{Id}_{\mathrm{mod}\ \Lambda}$.

\begin{lem} The following statements hold.
	\begin{itemize}
		\item[(a)] The functor  $\mathcal{R}$  is the torsion  radical  of the torsion pair
		  $( \mathrm{mod} \  \mathcal U, \mathrm{add}\ S)$ in $ \mathrm{mod}  \ \Lambda$.
		\item[(b)]  Let $X$ be a   $\Lambda-$module. The canonical sequence in this torsion pair
		\begin{equation}\label{eqcan}
			\xymatrix{0 \ar[r] & \mathcal{R}X \ar[r] & X \ar[r] & S^{r_X} \ar[r] & 0}  
		\end{equation}
		
		satisfies that  $r_X=\mathrm{dim}_K\mathrm{Hom}_{\Lambda}(X,S)$.
	\end{itemize}
\end{lem}

\begin{proof}

 \ (a) By that mentioned above, a  $\Lambda-$module $X$  can be seen as a 
 $\U-$ module if and only if  $ X\left(\begin{smallmatrix} T & 0 \\ M & U \end{smallmatrix}\right)\cong 
  X\left(\begin{smallmatrix} 0& 0 \\ M & U \end{smallmatrix}\right)$,  for all  $\left(\begin{smallmatrix} T & 0 \\ M & U \end{smallmatrix}\right)\in \Lambda$. Since $X\left(\begin{smallmatrix} 0 & 0 \\ M & U \end{smallmatrix}\right)= \mathcal{R}X \left(\begin{smallmatrix} T & 0 \\ M & U \end{smallmatrix}\right)$, we conclude that a   $\Lambda-$module $X$  can be seen as a 
 $\U-$ module if and only if $ X\cong\mathcal{R}X$. 
 
 On the other hand, $\footnotesize\mathcal{R}S=0$  since 
$(\mathcal{R}S)(U)=S \left(\begin{smallmatrix} 0 & 0 \\ M & U \end{smallmatrix}\right)=0$. Moreover, if   $\mathcal{R}X=0$, then $\footnotesize X\left(\begin{smallmatrix} 0 & 0 \\ M & U \end{smallmatrix}\right)=0$ for all  $\footnotesize U\in \U$, and therefore  $\footnotesize X\in \mathrm{add}\ S$. Thus $\footnotesize\mathcal{R}X=0$ if and only if $\footnotesize X\in \mathrm{add}\ S$, and it is clear that
$\mathcal R^2X=\mathcal R X$. Applying the exact functor $\mathcal R$ to the short exact sequence of $\Lambda$-modules $0\rightarrow \mathcal R X\rightarrow X\rightarrow X/\mathcal R X\rightarrow 0$
yields $\mathcal R(X/\mathcal R X)=0$.

(b) The statement follows after applying the functor   $\mathrm{Hom}_{\Lambda}(-,S)$  to the canonical sequence (\ref{eqcan}):
		\[
		\xymatrix@R-2pc{0 \ar[r] & \mathrm{Hom}_{\Lambda}(S^{r_X},S) \ar[r] & \mathrm{Hom}_{\Lambda}(X,S) \ar[r] & \mathrm{Hom}_{\Lambda}(\mathcal{R}X,S)= 0 \\
		\ar[r] & \mathrm{Ext}_{\Lambda}^1(S^{r_X},S) = 0.}
		\]
\end{proof}

\begin{cor}\label{pdim1}
For any $\Lambda$-module $X$, the $\mathcal U$-module $\mathcal R X$  is projective (in which case, $\mathrm{pd}_\Lambda X \le  1$) or else  
$\mathrm{pd}_\Lambda X  =\mathrm{pd}_\mathcal U\mathcal R X $.
\end{cor}
\begin{proof}
For all $U\in\mathcal U$, the $\mathcal U$-projective module $\mathcal U(U,-)$  can be identified under the full embedding of $(\mathcal U, \mathrm{mod}\ K)$ in $(\Lambda, \mathrm{mod}\ K) $, with the projective \hspace{1.5cm} $\Lambda$-module  $\left(\left(\begin{smallmatrix} 0 & 0 \\ M & U \end{smallmatrix}\right),-\right) $. If $\mathcal RX$ is  a projective  $\mathcal U$-module, then $\mathrm{pd}_\Lambda S\le 1$ implies $\mathrm{pd}_\Lambda X\le 1$. In other case, if  $\mathrm{ pd}_\mathcal U \mathcal R X=d$, thus  $\mathrm{pd}_\Lambda \mathcal RX=d$, and the sequence (\ref{eqcan}) yields $\mathrm{pd}_\Lambda X=d$. 
\end{proof}

\begin{lem}\label{lemaext}
	Let  $G$ be a  $\U-$module. There is an isomorphism of  $K-$vector   spaces $$\mathrm{Ext}^1_{\Lambda}(S,G)\cong \mathrm{Hom}_{\Lambda}(P_0,G).$$
\end{lem}

\begin{proof}
	By applying  $\mathrm{Hom}_{\Lambda}(-,G)$ to the exact sequence
	\[
	\xymatrix{0 \ar[r] & P_0 \ar[r] & P \ar[r] & S \ar[r] & 0},
	\]
	we get
	\[
	\xymatrix@R-2pc{0 \ar[r] & \mathrm{Hom}_{\Lambda}(S,G) \ar[r] & \mathrm{Hom}_{\Lambda}(P,G) \ar[r] & \mathrm{Hom}_{\Lambda}(P_0,G) \hspace{.1cm} \\	\ar[r] & \mathrm{Ext}_{\Lambda}^1(S,G) \ar[r] & \mathrm{Ext}_{\Lambda}^1(P,G)= 0\hspace{.1cm}.}
	\]
	Now the desired result follows from the fact
	  $$\tiny \mathrm{Hom}_{\Lambda}(P,G)=\mathrm{Hom}_{\Lambda}\left(\left(\begin{pmatrix} * & 0 \\ M & 0\end{pmatrix}, -\right), G\right)\cong G\left(\begin{pmatrix} * & 0 \\ M & 0\end{pmatrix}\right)=G \left(\begin{pmatrix} 0 & 0 \\ M & 0\end{pmatrix}\right)=0.$$
	Note that for all  $\U-$module $G$ we have
	\begin{equation}
	 \mathrm{Hom}_{\Lambda}(S,G)=0.
	\end{equation}
\end{proof}

 Since  $(\mathcal{R},\mathcal{E})$ is an adjoint pair of functors, there are, associated with it, a counit $\epsilon: \mathcal{R}\mathcal{E} \rightarrow \mathrm{Id}_{\mathrm{Mod}\ \U}$ and a unit $\delta: \mathrm{Id}_{\mathrm{Mod}\ \Lambda}\rightarrow \mathcal{E}\mathcal{R}$.
   
   We have the following result that extends to the  one given in  \cite{Assem1}.

\begin{pro}\label{cou}
	The adjoint pair  $(\mathcal{R},\mathcal{E})$ satisfies of following. 
	\begin{itemize}
		\item[(a)] The counit  $\epsilon$ is a functorial isomorphism.
		\item[(b)]  For all $\Lambda$-module  $X$,  the kernel and cokernel   $\delta_X$ lie in $\mathrm{add}\ S$.
		\item[(c)]  Let $X\in \mathrm{Mod}\ \Lambda$. The following are equivalent:
		\begin{itemize}
			\item[(i)] $\delta_X$ is a monomorphism,
			\item[(ii)] $S$ is not direct summand of  $X$,
			\item[(iii)] $\mathrm{Hom}_{\Lambda}(S,X)=0$.
		\end{itemize}
	\end{itemize}
\end{pro}

\begin{proof}
We only prove  (a). The rest is analogous to the proof given in \cite[Proposition 2.5]{Assem1}. It is clear that  $\epsilon$  is an isomorphism since, for all $U$, we have 
		\[
		(\mathcal{R}\mathcal{E}G)(U)=(\mathcal{E}G)\begin{pmatrix} 0 & 0 \\ M & U \end{pmatrix}\cong  \mathrm{Hom}_{\mathcal U}(\mathcal U(U,-), G)\cong G(U).
		\]
\end{proof}

It follows from the above result that  $\mathcal{E}$  is a fully faithful functor, \cite[Theorem (IV. 3.3)]{lane1971categories}.  The right perpendicular  category of $S$  is the full subcategory of 
$(\Lambda, \mathrm{mod} \ K)$  defined by 
$$\footnotesize S^{perp}=\{X \in (\Lambda, \mathrm{mod}\ K) \hspace{.1cm}|\hspace{.1cm} \mathrm{Hom}_{\Lambda}(S,X)=0,\mathrm{Ext}^1_{\Lambda}(S,X)=0\}.$$

\subsection { Homological properties of functors $(\mathcal R,\mathcal E)$}

It then follows that $\delta_X$ is a functorial isomorphism for all
 $X\in S^{perp}$; see \cite[Lemma 3.1]{Assem1}. Let $G$ be a $\mathcal U$-module. Thus, by Proposition \ref{cou}  there exists an exact sequence called \emph{the extension sequence}  
 \begin{equation}\label{extension}
 \footnotesize\xymatrix{0 \ar[r] & G \ar[r]^{\delta_G} & \mathcal{E}\mathcal{R}G \ar[r] & S^{e_G} \ar[r] & 0}, 
 \end{equation}
  which coincides with the restriction sequence for $\mathcal E RG\cong \mathcal E G$. In particular, $e_G=r_{\mathcal{ER}G}$.
 
\begin{pro}\label{Sperp}
  Let $G$ be a $\mathcal U$-module. The extension sequence  satisfies the following 
properties
\begin{itemize}
\item [(a)] $e_G=\mathrm{dim}_K\mathrm{Ext}_\Lambda ^1(S,G).$
\item [(b)] The connecting morphism
 $\mathrm{Hom}_\Lambda(S,S^{e_G})\rightarrow \mathrm{Ext}_\Lambda ^1(S,G)$ is an isomorphism.
\item [(c)] $\mathcal E G\in S^{perp}$.
 \end{itemize}
 \end{pro}
 
 \begin{proof}
 We only prove (a). The rest of the proof is similar to that given in \cite{Assem1}. 
 Evaluating the exact sequence (\ref{extension}) in  the object  $\left(\begin{smallmatrix} * & 0 \\ M & U \end{smallmatrix}\right)\in\Lambda$, we have
\[
\footnotesize\begin{tikzcd} 
	0 \ar[r] & G(U) \arrow[r] & \left(\restr{\C(*,-)}{\U} ,G\right)\coprod G(U) \arrow[r] & \left[S\begin{pmatrix} * & 0 \\ M & U \end{pmatrix}\right]^{e_G} \arrow[r] & 0 
\end{tikzcd} 
\]
since  $\mathcal{E}\mathcal{R}G\cong \mathcal{E}G$. By Lemma \ref{lemaext}, we obtain
$\footnotesize e_G=\mathrm{dim}_K\left(\restr{\C(*,-)}{\U} ,G\right)=\mathrm{dim}_KHom_{\Lambda}(P_0 ,G)=\mathrm{dim}_K\mathrm{Ext}^1_{\Lambda}(S,G)$. 
  \end{proof}

It follows that $\mathrm{mod}\ \U$ and $S^{perp}$  are equivalent categories.  
 \begin{cor}\label{equiv1} 
 The functors $\mathcal E$ and $\mathcal R$ induce an equivalence between $\mathrm{mod}\ \mathcal U$ 
 and $S^{perp}$,  and $\mathcal E$  is sends finitely presented $\mathcal U$-modules into finitely presented $\Lambda$-modules.
 \end{cor}
 
 \begin{proof}
 The first statement is a consequence of  Propositions \ref{cou} (a), \ref{Sperp} (c) and \cite[Lemma 3.1]{Assem1}.
 
 On the other hand, assume  that $G\in \mathrm{mod} \ \mathcal U$. $G$ is then seen as a $\Lambda$-module. Since 
 $\mathcal E\mathcal R G\cong \mathcal E G$, it follows from the exact sequence (\ref{extension})  that $\mathcal EG$ belongs to $\mathrm{mod} \ \Lambda$ since $S$  lies in $\mathrm{mod}\  \Lambda$.
 \end{proof}

The following result extends the one given in \cite{Assem1}, which has a similar proof.

\begin{pro}\label{extiso}
Let $X$ and $Y$ be in $\mathrm{mod} \ \Lambda$.Then:
\begin{itemize}
\item[(a)] There is an epimorphism $\mathrm{Ext}^1_\Lambda (X, Y )\rightarrow \mathrm{Ext}^1_\mathcal U (\mathcal R X, \mathcal R Y ) .$
\item[(b)] There is an isomorphism $\mathrm{Ext}^j_\Lambda (X, Y ) \cong \mathrm{Ext}^j_\mathcal U(\mathcal R X, \mathcal R Y )$, for
each $j \ge  2$.
\item[(c)] If $Y \in S^{perp}$, then the epimorphism of (a) is an isomorphism.

\end{itemize}
\end{pro}

\bigskip

\begin{proof}[Proof of Theorem \ref{mainExt} ]
The first part is a consequence of Proposition \ref{fp1} and Corollary \ref{equiv1}. 
Let $M$ and $N$ be orthogonal $\mathcal U$-modules. By Proposition \ref{Sperp}  (c), $\mathcal E N, \mathcal E M\in S^{perp}$ so that, by Corollary  \ref{equiv1},
$$\mathrm{Ext}^j_\Lambda(\mathcal EM,\mathcal EN)\cong \mathrm{Ext}^j_\mathcal U(\mathcal {RE}M,\mathcal {RE} N)\cong \mathrm{Ext}^j_\mathcal U(M,N)=0.$$ Thus, $\mathcal E M$ and $\mathcal EN$ are orthogonal.

Let $X$ and $Y$  be orthogonal $\Lambda$-modules. Therefore, Proposition \ref {extiso} (b) yields,  $\mathrm{Ext}^j_\Lambda(X,Y)\cong \mathrm{Ext}^j_\mathcal U(\mathcal {R}X, \mathcal{R}Y)\cong
 0$, for each $ j \ge  2$,  and  by \ref {extiso} (a) $\mathrm{Ext}^1_\Lambda( X,Y)$ $=0$ implies $ \mathrm{Ext}^1_\mathcal U(\mathcal {R}X, \mathcal{R}Y)\cong
 0$. Thus, $\mathcal R X$ and $\mathcal R Y$ are   orthogonal. The statement about exceptionality follows from Corollary \ref{pdim1}.
\end{proof}

\bigskip

We finish this section  showing with a couple of examples how we can extend classical tilting  categories in functor  categories of path categories by using the developed theory.

 We now recall the definition of classical tilting subcategories in functor categories. Recall that an \emph{annuli variety} is  an additive category where the idempotents split \cite{Aus}. The following definition is given in \cite{MVO1}.
 
\begin{defi}
	Let $\mathcal{C}$ be an annuli variety. A subcategory  $\mathcal{T}$ of $\mathrm{Mod}\ \mathcal{C}$  is a classical tilting category if the following holds:
	\begin{enumerate}
		\item[(i)] $\mathrm{pd}\ \mathcal{T}\leq 1$.
		\item[(ii)] $\mathrm{Ext}^1_{\mathcal{C}}(T_i,T_j)=0$, for all pair of objects $T_i,T_j\in \mathcal{T}$.
		\item[(iii)] For all object $C$ in $\mathcal{C}$, there exists an exact sequence  
		\[
		\xymatrix{0 \ar[r] & \mathcal{C}(C,-) \ar[r] & T_1 \ar[r] & T_2 \ar[r] & 0},
		\]
		with  $T_1,T_2\in \mathcal{T}$.
	\end{enumerate}
	
	A subcategory of $\mathrm{Mod}\ \C$ that satisfies (i) and (ii) is called partial tilting.
\end{defi}

\begin{ex}
Let  $\mathcal{T}$ be a classical tilting category in  $\mathrm{Mod}\ \U$.  Thus the full category  $\mathcal E(\mathcal T)$ of $\mathrm{Mod}\ \Lambda$ consisting of the objects $\mathcal ET$, $T\in\mathcal T$ can be extended to a classical tilting category in $\mathrm{Mod}\ \Lambda$. 
\bigskip

(i) Since $\mathrm{pd}\  T\le 1$ and $\mathrm{pd} \ S^{e_T}\le 1$,  it follows from the exact sequence
\[
\xymatrix{0 \ar[r] & T \ar[r] & \mathcal{E}T \ar[r] & S^{e_T} \ar[r] & 0}
\]
 that $\mathrm{pd}\  \mathcal ET\le 1$.

 (ii)  By Theorem   \ref{mainExt}, $\mathrm{Ext}^1_{\Lambda}(\mathcal{E}T,\mathcal{E}T')=0$, for every pair of objects $T,T'\in\mathcal T.$ Thus, $\mathcal{E}T$ is  a partial tilting category in $\mathrm{Mod}\ \Lambda$. Thus, by using the Bongartz argument \cite{Bongartz}, \cite[Theorem 7]{MVO1}, we can obtain a  classical tilting category in  $\mathrm{Mod}\ \Lambda$ from the partial tilting category $\mathcal E(\mathcal T)$ in  $\mathrm{Mod}\ \U$.   A similar argument  allows to obtain a classical tilting category $\mathcal R(\mathcal T)$ in $\mathrm{Mod}\ \U$  from a classical tilting category $\mathcal T$  in $\mathrm{Mod}\ \Lambda$  by using the restriction functor $\mathcal R$. 
 \end{ex}

\begin{ex}	Label the vertices of  $\mathbb{Z} A_{\infty}$ in the following way:
	\[
\tiny
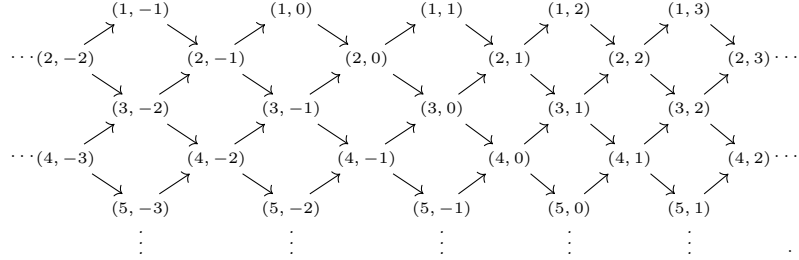
\begin{tikzcd}[cramped, row sep=1em, column sep=.15em]
	 & & (1,-1) \ar[rd]& & \hspace{.07cm} (1,0) \hspace{.07cm} \ar[rd]& & \hspace{.07cm} (1,1) \hspace{.07cm} \ar[rd]& & (1,2) \ar[rd]& & (1,3) \ar[rd]& &  \\ 
	\hspace{.01cm} & (2,-2) \arrow[l,phantom,"\cdots \hspace{.15cm}"] \ar[ru] \ar[rd]& & (2,-1) \ar[ru] \ar[rd]& & \hspace{.07cm} (2,0) \hspace{.07cm} \ar[ru] \ar[rd]& & (2,1) \ar[ru] \ar[rd]& & (2,2) \ar[ru] \ar[rd]& & (2,3) \arrow[r,phantom,"\hspace{.15cm} \cdots"] & \hspace{.01cm}  \\ 
	& & (3,-2) \ar[rd] \ar[ru]& & (3,-1) \ar[rd] \ar[ru]& & \hspace{.07cm} (3,0) \hspace{.07cm} \ar[rd] \ar[ru]& & (3,1) \ar[rd] \ar[ru]& & (3,2) \ar[rd] \ar[ru]& &   \\ 
	\hspace{.01cm} & (4,-3) \arrow[l,phantom,"\cdots \hspace{.15cm}"] \ar[ru] \ar[rd]& & (4,-2) \ar[ru] \ar[rd]& & (4,-1) \ar[ru] \ar[rd]& & (4,0) \ar[ru] \ar[rd]& & (4,1) \ar[ru] \ar[rd]& &  (4,2) \arrow[r,phantom,"\hspace{.15cm} \cdots"] & \hspace{.01cm}  \\ 
	& & (5,-3)\ar[ru] \arrow[d,phantom,"\vdots"] & & (5,-2) \ar[ru] \arrow[d,phantom,"\vdots"] & & (5,-1) \ar[ru] \arrow[d,phantom,"\vdots"] & & (5,0) \ar[ru] \arrow[d,phantom,"\vdots"] & & (5,1) \ar[ru] \arrow[d,phantom,"\vdots"] & &  \\ 
	& & \hspace{.01cm} & & \hspace{.01cm} & & \hspace{.01cm} & & \hspace{.01cm} & & \hspace{.01cm} & & . 
\end{tikzcd}
\]

	Let $(r,s)\in\mathbb{N}\times \mathbb Z$. Define the representation
	$T{(r,s)}=(V_{ij}, f_\alpha)_{(i,j)\in \mathbb{N}\times \mathbb Z} $, given by
	\begin{displaymath}
		V_{ij} = \left\{ \begin{array}{ll}
			K & \mbox{if $i\geq r$ and $-(i-r-s)\leq j\leq s$;} \\
			0 & \mbox{in other case.} 
		\end{array}
		\right.
	\end{displaymath}
and $f_\alpha=1_K: V_{rt}\rightarrow  V_{uv}$ if $V_{rt}= V_{uv}=K$ and 
$f_\alpha=0$ in other case:

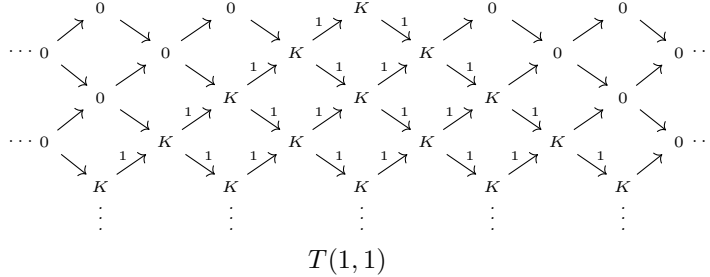
\[
\tiny\begin{tikzcd}[cramped, row sep=1em, column sep=.1em]
	& & \hspace{.22cm} 0 \hspace{.22cm} \ar[rd]& & \hspace{.22cm} 0 \hspace{.22cm} \ar[rd]& & \hspace{.22cm} K \hspace{.22cm} \ar[rd,"1"]& & \hspace{.22cm} 0 \hspace{.22cm} \ar[rd]& & \hspace{.22cm} 0 \hspace{.22cm} \ar[rd]& &  \\ 
	\hspace{.51cm} & \hspace{.15cm} 0 \hspace{.15cm} \arrow[l,phantom,"\cdots"] \ar[ru] \ar[rd]& & \hspace{.22cm} 0 \hspace{.22cm} \ar[ru] \ar[rd]& & \hspace{.22cm} K \hspace{.22cm} \ar[ru,"1"] \ar[rd,"1"]& & \hspace{.22cm} K \hspace{.22cm} \ar[ru] \ar[rd,"1"]& & \hspace{.22cm} 0 \hspace{.22cm} \ar[ru] \ar[rd]& & \hspace{.15cm} 0 \hspace{.15cm} \arrow[r,phantom,"\cdots"] & \hspace{.01cm}  \\ 
	& & \hspace{.22cm} 0 \hspace{.22cm} \ar[rd] \ar[ru]& & \hspace{.22cm} K \hspace{.22cm} \ar[rd,"1"] \ar[ru,"1"]& & \hspace{.22cm} K \hspace{.22cm} \ar[rd,"1"] \ar[ru,"1"]& & \hspace{.22cm} K \hspace{.22cm} \ar[rd,"1"] \ar[ru]& & \hspace{.22cm} 0 \hspace{.22cm} \ar[rd] \ar[ru]& &   \\ 
	\hspace{.51cm} & \hspace{.15cm} 0 \hspace{.15cm} \arrow[l,phantom,"\cdots"] \ar[ru] \ar[rd]& & \hspace{.22cm} K \hspace{.22cm} \ar[ru,"1"] \ar[rd,"1"]& & \hspace{.22cm} K \hspace{.22cm} \ar[ru,"1"] \ar[rd,"1"]& & \hspace{.22cm} K \hspace{.22cm} \ar[ru,"1"] \ar[rd,"1"]& & \hspace{.22cm} K \hspace{.22cm} \ar[ru] \ar[rd,"1"]& & \hspace{.15cm} 0 \hspace{.15cm} \arrow[r,phantom,"\cdots"] & \hspace{.01cm}  \\ 
	& & \hspace{.22cm} K \hspace{.22cm} \ar[ru,"1"] \arrow[d,phantom,"\vdots"] & & \hspace{.22cm} K \hspace{.22cm} \ar[ru,"1"] \arrow[d,phantom,"\vdots"] & & \hspace{.22cm} K \hspace{.22cm} \ar[ru,"1"] \arrow[d,phantom,"\vdots"] & & \hspace{.22cm} K \hspace{.22cm} \ar[ru,"1"] \arrow[d,phantom,"\vdots"] & & \hspace{.22cm} K \hspace{.22cm} \ar[ru] \arrow[d,phantom,"\vdots"] & &  \\ 
	& & \hspace{.01cm} & & \hspace{.01cm} & & \hspace{.01cm} & & \hspace{.01cm} & & \hspace{.01cm} & & . \\
\end{tikzcd}
\]
\vspace{-.3cm}
\[\tiny{T(1,1)}\]

Consider the path category  $K(\mathbb{Z} A_{\infty},\sigma)$ of the quiver $\mathbb{Z} A_{\infty}$ modulo the mesh ideal; see \cite{Rin2}, for more details.

	In \cite[Theorem 5.4]{Martin}, it is proved that  $\mathrm{add} \{T(r,s)\}_{(r,s)\in \mathbb{N}\times \mathbb{Z}}$ is a classical tilting category in  $\mathrm{Mod}(K(\mathbb{Z} A_{\infty},\sigma))$. Consider the path category $\Lambda=K(\mathbb{Z} A_{\infty}\cup \{*\rightarrow (1,1)\},\sigma)$, modulo the ideal generated by the mesh relations. Thus, $\Lambda\cong\begin{pmatrix} * & 0 \\ M & K(\mathbb{Z} A_{\infty},\sigma) \end{pmatrix}$, where 
	$M=K(\mathbb{Z} A_{\infty},\sigma)\otimes\mathrm{add}\{*\}\rightarrow \mathrm{mod} \ K$ is given by
	$M(U,*)=\Lambda(*,U)$ for all vertex $U$ in $\mathbb{Z} A_{\infty}$: 

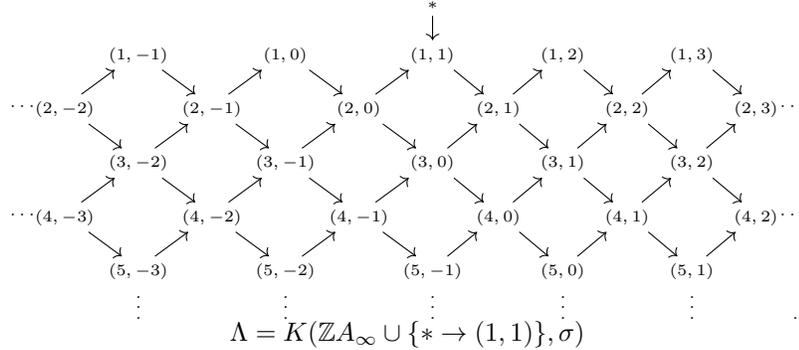
	\[
\tiny\begin{tikzcd}[cramped, row sep=1.2em, column sep=.07em]
	& & & & & & * \ar[d]& & \hspace{.75cm} & & \hspace{.75cm} & &  \\ 
	& & (1,-1) \ar[rd]& & \hspace{.07cm} (1,0) \hspace{.07cm} \ar[rd]& & \hspace{.07cm} (1,1) \hspace{.07cm} \ar[rd]& & (1,2) \ar[rd]& & (1,3) \ar[rd]& &  \\ 
	\hspace{.01cm} & (2,-2) \arrow[l,phantom,"\cdots\hspace{.15cm}"] \ar[ru] \ar[rd]& & (2,-1) \ar[ru] \ar[rd]& & \hspace{.07cm} (2,0) \hspace{.07cm} \ar[ru] \ar[rd]& & (2,1) \ar[ru] \ar[rd]& & (2,2) \ar[ru] \ar[rd]& & (2,3) \arrow[r,phantom,"\hspace{.15cm}\cdots"] & \hspace{.01cm}  \\ 
	& & (3,-2) \ar[rd] \ar[ru]& & (3,-1) \ar[rd] \ar[ru]& & \hspace{.07cm} (3,0) \hspace{.07cm} \ar[rd] \ar[ru]& & (3,1) \ar[rd] \ar[ru]& & (3,2) \ar[rd] \ar[ru]& &   \\ 
	\hspace{.01cm} & (4,-3) \arrow[l,phantom,"\cdots\hspace{.15cm}"] \ar[ru] \ar[rd]& & (4,-2) \ar[ru] \ar[rd]& & (4,-1) \ar[ru] \ar[rd]& & (4,0) \ar[ru] \ar[rd]& & (4,1) \ar[ru] \ar[rd]& &  (4,2) \arrow[r,phantom,"\hspace{.15cm}\cdots"] & \hspace{.01cm}  \\ 
	& & (5,-3)\ar[ru] \arrow[d,phantom,"\vdots"] & & (5,-2) \ar[ru] \arrow[d,phantom,"\vdots"] & & (5,-1) \ar[ru] \arrow[d,phantom,"\vdots"] & & (5,0) \ar[ru] \arrow[d,phantom,"\vdots"] & & (5,1) \ar[ru] \arrow[d,phantom,"\vdots"] & &  \\ 
	& & \hspace{.01cm} & & \hspace{.01cm} & & \hspace{.01cm} & & \hspace{.01cm} & & \hspace{.01cm} & & . 
\end{tikzcd}
\]
\vspace{-.3cm}
\[\tiny{\Lambda=K(\mathbb{Z} A_{\infty}\cup \{*\rightarrow (1,1)\},\sigma)}\]

	We see that $\{\mathcal{E}T(r,s)\coprod S\}_{(r,s)\in \mathbb{N}\times \mathbb{Z}}$ is a classical tilting category in $\mathrm{Mod}\ \Lambda$. $(i)$ Since $\mathrm{pd}S\leq 1$, it follows that $\mathrm{pd}(\mathcal{E}T(r,s)\coprod S)\leq 1$.
	
	$(ii)$ Let $(r,s),(r',s')\in \mathbb{N}\times \mathbb{Z}$.  First note that by Corollary 3.5 we have
	 $\mathrm{Ext}^1_{\Lambda}(S, \mathcal{E}T(r',s'))\cong$ $ 
	 \mathrm{Ext}^1_{\Lambda}(\mathcal RS, T(r',s'))=0$ because
	 $\mathcal RS=0$. Moreover, since  $\mathcal{E}$ preserves orthogonality  and $S$ is an injective $\Lambda-$module, we have  $$\mathrm{Ext}^1_{\Lambda}(\mathcal{E}T(r,s)\coprod S,\mathcal{E}T(r',s')\coprod S)=0.$$
	
	$(iii)$ 
	Set $E_j^i$ the vertex $(i,j)\in (\mathbb{Z} A_{\infty})_0$. Note that $\restr{\Lambda(*,-)}{K(\mathbb{Z} A_{\infty},\sigma)}\cong K(\mathbb{Z} A_{\infty},\sigma)(E_1^1,-)$. Let $(r,s)\in \mathbb{N}\times \mathbb{Z}$. By Yoneda's lemma, we have isomorphisms  $(\mathcal{E}T(1,1))\begin{pmatrix} 0 & 0 \\ M & E_s^r \end{pmatrix}\cong T(1,1)(E_s^r)$ and $(\mathcal{E}T(1,1))\begin{pmatrix} * & 0 \\ M & 0 \end{pmatrix}\cong T(1,1)(E_1^1)\cong K$; therefore, $\mathcal{E}T(1,1)\cong T(1,1)\coprod S$. On the other hand, if   $(r,s)\neq (1,1)$, one can see that $\mathcal ET(r,s)\cong T(r,s)$. Thus, $T(1,1)$ is the unique tilting object that suffers a change under the functor $\mathcal E$.
	
	For  each $(r,s)\in \mathbb{N}\times\mathbb{Z}$, we can see $T(r,s)$ as an object in $\mathrm{Mod}\ \Lambda$, and we have a resolution in  $\mathrm{Mod}\ \Lambda$ of the form
	\[
	\xymatrix{0\ar[r]&(E_s^r,-) \ar[r]&T(1,r+s-1)\ar[r]&T(r+1,s-1)\ar[r]&0}.
	\]
	In addition, for  projective  $(*,-)$ we have the  resolution
	\[
	\xymatrix{0\ar[r]&(*,-) \ar[r]&T(1,1)\coprod S\ar[r]&T(2,0)\ar[r]&0}.
	\]
	
\[
\begin{tikzcd}[
	ampersand replacement=\&,
	execute at end picture={\scoped[on background layer]
\fill[pattern=north east lines, pattern color=blue]  (a.center) -- (b.center) -- (c.center) -- cycle;}, execute at end picture={\scoped[on background layer]\fill[pattern=north east lines, pattern color=blue]  (e.center) -- (b.center) -- (c.center) -- cycle;}, execute at end picture={\scoped[on background layer]\fill[pattern=north east lines, pattern color=blue]  (d.center) -- (b.center) -- (e.center) -- cycle;}, execute at end picture={\scoped[on background layer]
\fill[pattern=north east lines, pattern color=blue]  (f.center) -- (e.center) -- (c.center) -- cycle;}, execute at end picture={\scoped[on background layer]\fill[pattern=north east lines, pattern color=blue]  (d.center) -- (e.center) -- (h.center) -- cycle;}, execute at end picture={\scoped[on background layer]
\fill[pattern=north east lines, pattern color=blue]  (e.center) -- (f.center) -- (i.center) -- cycle;}, execute at end picture={\scoped[on background layer]\fill[pattern=north east lines, pattern color=blue]  (d.center) -- (g.center) -- (h.center) -- cycle;}, execute at end picture={\scoped[on background layer]
\fill[pattern=north east lines, pattern color=blue]  (e.center) -- (h.center) -- (i.center) -- cycle;},	execute at end picture={
\scoped[on background layer]\fill[pattern=north east lines, pattern color=blue]  (f.center) -- (i.center) -- (j.center) -- cycle;}, row sep=1.2em, column sep=.05em]
\& \& \& \& \& \& K \ar[d,"1"] \& \& \hspace{.75cm} \& \& \hspace{.75cm} \& \&  \\ 
\& \& \hspace{.22cm} 0 \hspace{.22cm} \ar[rd]\& \& \hspace{.22cm} 0 \hspace{.22cm} \ar[rd]\& \& \hspace{.22cm} K \hspace{.22cm} \ar[rd,"1"]\& \& \hspace{.22cm} 0 \hspace{.22cm} \ar[rd]\& \& \hspace{.22cm} 0 \hspace{.22cm} \ar[rd]\& \&  \\ 
\hspace{.01cm} \& \hspace{.15cm} 0 \hspace{.15cm} \arrow[l,phantom,"\cdots"] \ar[ru] \ar[rd]\& \& \hspace{.22cm} 0 \hspace{.22cm} \ar[ru] \ar[rd]\& \& |[alias=a]|\hspace{.22cm} K \hspace{.22cm} \ar[ru,"1"] \ar[rd,"1"]\& \& \hspace{.22cm} K \hspace{.22cm} \ar[ru] \ar[rd,"1"]\& \& \hspace{.22cm} 0 \hspace{.22cm} \ar[ru] \ar[rd]\& \& \hspace{.15cm} 0 \hspace{.15cm} \arrow[r,phantom,"\cdots"] \& \hspace{.01cm}  \\ 
\& \& \hspace{.22cm} 0 \hspace{.22cm} \ar[rd] \ar[ru]\& \& |[alias=b]|\hspace{.22cm} K \hspace{.22cm} \ar[rd,"1"] \ar[ru,"1"]\& \& |[alias=c]|\hspace{.22cm} K \hspace{.22cm} \ar[rd,"1"] \ar[ru,"1"]\& \& \hspace{.22cm} K \hspace{.22cm} \ar[rd,"1"] \ar[ru]\& \& \hspace{.22cm} 0 \hspace{.22cm} \ar[rd] \ar[ru]\& \&   \\ 
\hspace{.01cm} \& \hspace{.15cm} 0 \hspace{.15cm} \arrow[l,phantom,"\cdots"] \ar[ru] \ar[rd]\& \& |[alias=d]|\hspace{.22cm} K \hspace{.22cm} \ar[ru,"1"] \ar[rd,"1"]\& \& |[alias=e]|\hspace{.22cm} K \hspace{.22cm} \ar[ru,"1"] \ar[rd,"1"]\& \& |[alias=f]|\hspace{.22cm} K \hspace{.22cm} \ar[ru,"1"] \ar[rd,"1"]\& \& \hspace{.22cm} K \hspace{.22cm} \ar[ru] \ar[rd,"1"]\& \& \hspace{.15cm} 0 \hspace{.15cm} \arrow[r,phantom,"\cdots"] \& \hspace{.01cm}  \\ 
\& \& |[alias=g]|\hspace{.22cm} K \hspace{.22cm} \ar[ru,"1"] \arrow[d,phantom,"\vdots"] \& \& |[alias=h]|\hspace{.22cm} K \hspace{.22cm} \ar[ru,"1"] \arrow[d,phantom,"\vdots"] \& \& |[alias=i]| \hspace{.22cm} K \hspace{.22cm} \ar[ru,"1"] \arrow[d,phantom,"\vdots"] \& \& |[alias=j]|\hspace{.22cm} K \hspace{.22cm} \ar[ru,"1"] \arrow[d,phantom,"\vdots"] \& \& \hspace{.22cm} K \hspace{.22cm} \ar[ru] \arrow[d,phantom,"\vdots"] \& \&  \\ 
\& \& \hspace{.01cm} \& \& \hspace{.01cm} \& \& \hspace{.01cm} \& \& \hspace{.01cm} \& \& \hspace{.01cm} \& \& 
\end{tikzcd}
\]
\[\tiny{	\xymatrix{0\ar[r]&(*,-) \ar[r]&T(1,1)\coprod S\ar[r]&T(2,0)\ar[r]&0}
}\]
\end{ex}

\footnotesize

\vskip3mm \noindent Rafael Francisco Ochoa de la Cruz:\\ Instituto de Matem\'aticas, Universidad Nacional Aut\'onoma de M\'exico\\
Circuito Exterior, Ciudad Universitaria,
C.P. 04510, M\'exico, D.F. MEXICO.\\ {\tt rafaelfochoa88@gmail.com}

\vskip3mm \noindent Martin Ort\'iz Morales:\\ Facultad de Ciencias, Universidad  Aut\'onoma del Estado de M\'exico\\
Campus    Universitario ``El Cerrillo, Piedras Blancas'', Carretera   Toluca-Ixtlahuaca     Km.   15.5, Estado   de   M\'exico. CP 50200.
\\ {\tt mortizmo@uaemex.mx}

\end{document}